\documentclass[reqno,centertags,11pt,a4paper]{amsart}
\usepackage{amssymb,mathrsfs}
\usepackage{color,umoline}
\usepackage[dvipsnames]{xcolor}
\usepackage{graphicx}
\usepackage{enumitem}
\usepackage[font=small,labelfont=bf]{caption}
\usepackage{tikz, caption, subcaption}
\usepackage{relsize}
\usepackage[mathscr]{eucal}
\usepackage{verbatim}
\usepackage[draft]{todonotes}
\usepackage{hyperref}
\usepackage{amsthm}
\usepackage[dvipsnames]{xcolor}
\usetikzlibrary{patterns}
\tikzstyle{EDR}=[draw=lightgray,line width=0pt,preaction={clip, postaction={pattern=north east lines, pattern color=gray}}]
\tikzstyle{EDR1}=[draw=lightgray,line width=0pt,preaction={clip, postaction={pattern=north west lines, pattern color=gray}}]

\definecolor{mygray}{gray}{0.95}

\definecolor{mypink1}{rgb}{1.2,1.1,0.9}

\definecolor{mypink2}{rgb}{1.0,0.95 ,0.9}

\definecolor{mypink3}{rgb}{1.0,0.6,0.7}

\numberwithin{equation}{section}

\newtheorem{theorem}{Theorem}[section]
\newtheorem{definition}[theorem]{Definition}
\newtheorem{lemma}[theorem]{Lemma}

\newtheorem{proposition}[theorem]{Proposition}
\newtheorem{remark}[theorem]{Remark}

\numberwithin{equation}{section}

\newcommand{\supp}{\operatorname{Supp}}

\newcommand{\beq}{\begin{equation}}
	\newcommand{\eeq}{\end{equation}}
\newcommand{\beqq}{\begin{equation*}}
	\newcommand{\eeqq}{\end{equation*}}
\newcommand{\ben}{\begin{eqnarray}}
	\newcommand{\een}{\end{eqnarray}}
\newcommand{\beno}{\begin{eqnarray*}}
	\newcommand{\eeno}{\end{eqnarray*}}

\begin{document}
\setcounter{page}{1}
\title[Nc circular maximal inequality and local smoothing estimate]
{Noncommutative Bourgain's circular maximal theorem and a local smoothing estimate on the generalized Moyal planes}

\author[]{Guixiang Hong}
\address{
Institute for Advanced Study in Mathematics\\
Harbin Institute of Technology\\
Harbin
150001\\
China}
\email{gxhong@hit.edu.cn}

\author[]{Xudong Lai}
\address{
Institute for Advanced Study in Mathematics\\
Harbin Institute of Technology\\
Harbin
150001\\
China;
Zhengzhou Research Institute\\
Harbin Institute of Technology\\
Zhengzhou
450000\\
China}

\email{xudonglai@hit.edu.cn}
\author[]{Liang Wang}
\address{
School of Mathematics and Statistics\\
Wuhan University\\
Wuhan 430072\\
China;
Department of Mathematics\\
 City University of Hong Kong\\
 Hong Kong SAR\\
  China
}
\email{wlmath@whu.edu.cn}

\thanks{}


\date{\today
}
\begin{abstract}
 In this paper, we establish a local smoothing estimate on the generalized Moyal plane, also known as the two-dimensional quantum Euclidean space. This is a quantum version of the local smoothing estimate for the free wave equation due to Mockenhaupt$-$Seeger$-$Sogge \cite{MSS}. As an application and simultaneously one motivation, we obtain a noncommutative analogue of Bourgain's circular maximal theorem, resolving the problem after \cite{Hong}.
  \end{abstract}

\maketitle
\section{Introduction}

In 1976, Stein \cite{St76} introduced the spherical means on the Euclidean spaces and showed the surprising maximal inequality on $L_p(\mathbb R^d)$ for $d\geq 3$ and $\frac{d}{d-1}<p\leq\infty$, and the latter range is sharp in the sense that the maximal inequality cannot be true for $p\leq \frac{d}{d-1}$. Note that when $d=2$, $\frac{d}{d-1}=2$ and the circular maximal inequality cannot be true on $L_2(\mathbb R^d)$. This presented a significant obstacle to apply the $L_2$-orthogonality to extend Stein's results for circular means. Ten years later, Bourgain \cite{Bou85} (see \cite{Bo86} for a simpler but more general version) was able to surmount this obstacle by exploring interpolation, deep geometric consideration and variants of stopping time arguments, obtaining the circular maximal inequality and resolving the open problem left in \cite{St78}.

Motivated by the noncommutative martingale and ergodic maximal inequality \cite{Jun02, JX07, Mei09}, Stein's spherical maximal inequality was established in the noncommutative framework more than ten years ago by the first author and applied to the dimension-free operator-valued Hardy-Littlewood maximal inequality \cite{Hong}, see e.g. \cite{Pis98, Jun02, Anantharaman, CR, CXY, HJP,HLX,HLW, HRW, HWW} for more information on noncommutative maximal inequalities. However, Bourgain's arguments \cite{Bou85,Bo86} seem to be too difficult to admit a noncommutative analogue and thus the noncommutative version of circular maximal inequality had been left unsolved  before. In the classical setting, after Sogge's breakthrough \cite{Sogge92}, Muckenhoupt, Seeger and Sogge \cite{MSS} had provided an alternative proof of Bourgain's circular maximal inequality by exploring its close connection to local smoothing estimate, Bochner-Riesz means, Fourier restriction estimate and Kakeya's maximal inequality in two dimensional Euclidean spaces. See e.g. \cite{Tao} for more tight relationships among these seemingly unrelated topics. This leads us to study  these topics in the noncommutative setting by first focusing on quantum Euclidean spaces.

Quantum Euclidean spaces, the generalized Moyal planes or phase spaces, are the model examples of noncommutative locally compact manifolds, having appeared frequently in the literature of mathematical physics, such as string theory and noncommutative field theory \cite{DN01, NS98, SW99}.  In recent years, harmonic analysis on these noncommutative manifolds has been developing very rapidly, including the pointwise convergence of Fourier series and functional spaces on quantum tori \cite{CXY, Lai21, Xiong18},  singular integral theory and  pseudodifferential operator
theory \cite{GJM, GJP, JPMX}, as well as commutator estimates and quantum differentiability for quantum Euclidean spaces \cite{LSZ, Xiong1}. Note that all these work can be viewed as theories insider Calder\'on-Zygmund's framework that are closely related to {\it Littlewood-Paley square functions and noncommutative Hardy-Littlewood maximal inequalities}, but it is well-known that the previously mentioned topics around Fourier restriction estimates go dramatically different from Calder\'on-Zygmund theory. This provides us another impetus to investigate the local smoothing estimates of wave equation and the theory of partial differential equations on quantum Euclidean spaces.

 Along this research line, we have made some progresses \cite{FHW23,HLW23,Lai21}, including the sharp estimate of the Bochner-Riesz means on two-dimensional quantum tori, the Fourier restriction estimates with optimal indices on two-dimensional quantum Euclidean spaces and the noncommutative sharp endpoint $L_p$-estimates of Schr\"odinger equations; see \cite{ CGS, CHWW1,  CHWW2,M23,Rosenberg} for other progresses on partial differential equations.

 In this paper, we study the local smoothing estimates of quantum wave equations, obtaining the quantum Euclidean space analogues of the above-mentioned Muckenhoupt-Seeger-Sogge's results and resolving the problem of noncommutative circular maximal inequality.

The quantum Euclidean spaces admit  several equivalent definitions (see e.g. \cite{GIV,GJP, Re}). In this paper, we adopt the following one. Given a $d\times d $ real antisymmetric matrix $\theta$ with $d\in \mathbb N$, the quantum Euclidean space  $\mathcal{R}_\theta^d$ is a von Neumann subalgebra of $\mathcal{B}(L_2(\mathbb{R}^d))$ which is generated by  a family of unitary operators $\{U_\theta(r)\}_{r\in\mathbb{R}^d}$, where $\{U_\theta(r)\}_{r\in\mathbb{R}^d}$ satisfies the following Weyl relation:
\begin{equation*}
  U_\theta(r)U_\theta(s)=e^{\frac i2(s,\theta r)}U_\theta(r+s),\quad \text{for\: all}\: r,s\in\mathbb{R}^d.
\end{equation*}
The space $\mathcal{R}_\theta^d$ is a semifinite von Neumann algebra with a normal semifinite faithful trace $\tau_\theta$, and we denote by $L_p(\mathcal{R}_\theta^d)$ the noncommutative $L_p$ space associated to $(\mathcal{R}_\theta^d, \tau_\theta)$. When $\theta=0$, $L_p(\mathcal{R}_\theta^d)$ reduces to the usual $L_p$ space defined on $\mathbb{R}^d$ with the Lebesgue measure. As in the case $\theta=0$, there exists a canonical Laplacian $\Delta_\theta$ on $\mathcal R^d_\theta$. We refer the reader to Section \ref{S2} for more information and notions appearing below.

Let $u:\mathbb{R}_+\rightarrow \mathcal{R}_\theta^d$ be the solution to the Cauchy problem of  the wave equation
\begin{eqnarray}\label{weqq}
\left\{
\begin{array}{ll}
   (\partial_{tt}-\Delta_\theta)u=0, \  t\in \mathbb{R}_+,  \\[5pt]
 u(0)=x_0,\\[5pt]\partial_tu(0)=x_1
\end{array}
\right.
\end{eqnarray}
where $x_0,x_1\in\mathcal{S}(\mathcal{R}_\theta^d)$, the Schwartz class on $\mathcal{R}_\theta^d$.

As in the classical case $\theta=0$, one may represent the solution of \eqref{weqq} as Fourier multipliers,
\begin{equation*}
  u(t)=\int_{\mathbb{R}^d}\cos(|t\xi|){\hat{x}_0}(\xi)U_\theta(\xi)d\xi+ \int_{\mathbb{R}^d}\sin(|t\xi|)\frac{{\hat{x}_1}(\xi)}{|\xi|}U_\theta(\xi)d\xi,
\end{equation*}
where $\hat{x}(\xi)=\tau_\theta(xU^*_\theta(\xi))$.

For $\theta=0$ and a fixed time $t$,  Miyachi \cite{Miyachi80} and Peral \cite{Per80} showed the following sharp $L_p$ estimate for $p\in(1,\infty)$,
\begin{equation}\label{PM}
   \|u(t)\|_{L_{p}(\mathbb{R}^d)}\leq C_{p,t}\big(\|x_0\|_{L_{p, s_p}(\mathbb{R}^d)}+ \|x_1\|_{L_{p, s_p-1}(\mathbb{R}^d)}\big),
\end{equation}
where $s_p=(d-1)|\frac{1}{2}-\frac{1}{p}|$, $C_{p,t}$ is locally bounded and $\|\cdot\|_{L_{p, s}}$ denotes the inhomogeneous Sobolev norm with $s$ derivatives (see e.g. \cite[Page 13]{GF}). This estimate is sharp in the sense that the inequality \eqref{PM} fails if $s_p$ is replaced by any $\sigma$ with $\sigma<s_p$.

For general $\theta$, we get the same fixed time $L_p$ estimate.

\begin{proposition}\label{prop:fixed estimate}
With all the notions above,
\begin{equation}\label{PM2}
 \|u(t)\|_{L_{p}(\mathcal{R}_\theta^d)}\leq C_{p,t}\big(\|x_0\|_{L_{p,s_p}(\mathcal{R}_\theta^d)}+ \|x_1\|_{L_{p,s_p-1}(\mathcal{R}_\theta^d)}\big).
\end{equation}

\end{proposition}

Note that as in the classical case the estimate \eqref{PM2} cannot be deduced from noncommutative H\"ormander-Mihlin Fourier multiplier theorems or Calder\'on-Zygmund theory. We will first explore a transference technique to reduce \eqref{PM2} to an operator-valued one, and then use an observation made by Mei \cite{Mei09} which says that any Fourier multiplier on classical Hardy spaces $H_1(\mathbb R^d)$ is automatically completely bounded. This, combined with the noncommutative  Hardy-BMO theory and analytic interpolation, will conclude the proof. It is worthwhile to mention that our transference technique for $\mathbb R^d$ is different from the ones used in the literature (cf. \cite{CPPR15, CP21}) for discrete groups, and thus needs some new arguments; actually, there are some open problems regarding the similar transference techniques for non-discrete groups, see e.g. Remark \ref{rem:trans}, \cite[Remark 3.5]{CP21}. Since it is viewed as a motivating result of the main local smoothing estimate of the present paper, we will prove it in the Appendix.

By the inequality \eqref{PM2}, one then has trivially
\begin{equation}\label{e:trilocal}
 \Big(\int_1^2 \|u(t)\|_{L_{p}(\mathcal{R}_\theta^d)}^pdt\Big)^{1/p}\lesssim\|x_0\|_{L_{p,s_p}(\mathcal{R}_\theta^d)}+ \|x_1\|_{L_{p,s_p-1}(\mathcal{R}_\theta^d)},
\end{equation}
where $\|\cdot\|_{L_{p,s}(\mathcal{R}_\theta^d)}$ is the Sobolev norm. A natural problem here is that if we consider the time-space estimates \eqref{e:trilocal} directly, can we weaken the regularities requirement of the initial data $x_0$ and $x_1$ for the wave equation \eqref{weqq}.
The first main result of the present paper is the following non-trivial local smoothing estimate on $\mathcal{R}_\theta^2$ for all $2\times 2 $ real  antisymmetric matrix $\theta$.

\begin{theorem}\label{qeloc}
Let $\theta$ be a $2\times 2 $ real  antisymmetric matrix and assume that $2<p<\infty$. Then for all $\kappa<\kappa(p)$, we have
\begin{equation}\label{loc}
  \|u\|_{L_p(L_\infty([1,2])\overline{\otimes}\mathcal{R}_\theta^2)}\leq C_{p,\kappa}\big(\|x_0\|_{L_{p,s_p-\kappa}(\mathcal{R}_\theta^2)}+\|x_1\|_{L_{p,s_p-\kappa-1}(\mathcal{R}_\theta^2)}\big),
\end{equation}
where $u(t)$ is the solution of the equation \eqref{weqq}, $C_{p,\kappa}$ is a constant that depends on $p,\kappa$, and
\begin{eqnarray*}
 \kappa(p)=\left\{
\begin{array}{ll}
    1/4-1/2p\ , \quad \ 2<p<4\ , \\[5pt]
 1/2p\ ,\quad\quad\quad\quad 4\leq p<\infty.
\end{array}
\right.
\end{eqnarray*}
\end{theorem}

This is the quantum Euclidean space analogue of Muckenhoupt-Seeger-Sogge's celebrated local smoothing estimate. Note that the latter has been playing an important role in motivating the striking progress on the local smoothing conjecture (see e.g. \cite{BD, GWZ20, Wo00}), which was posed by Sogge \cite{Sogge92} and is still open in the higher-dimensions $d\geq3$ with a recent resolution in the two dimensions \cite{GWZ20}. With Theorem \ref{qeloc}, it would be quite interesting to establish the noncommutative analogues of these progresses \cite{BD, Wo00}, especially, the quantum Euclidean space analogue of \cite{GWZ20}---\eqref{loc} with $\kappa<2\kappa(p)$. This turns out to be extremely challenging and will be taken care of elsewhere.

Via the above-mentioned non-standard transference  technique (see Section \ref{S6} for details), Theorem \ref{qeloc} reduces to the following operator-valued local smoothing estimate. Let $\mathcal{M}$ denote a semifinite von Neumann algebra equipped with semifinite normal faithful trace $\tau$. The associated dense ideal of elements with finite trace and $L_p$ spaces are denoted respectively by $\mathcal S(\mathcal M)$ and $L_p(\mathcal M)$.
Consider the tensor von Neumann algebra $L_\infty([1,2]\times\mathbb{R}^2)\overline{\otimes}\mathcal{M}$. For $f,g\in\mathcal{S}(\mathbb{R}^2)\otimes\mathcal{S}(\mathcal{M})$, one can define $u\in L_\infty([1,2]\times\mathbb{R}^2)\overline{\otimes}\mathcal{M}$ by \begin{equation*}
  u(x,t):=\int_{\mathbb{R}^d}e^{ix\xi}\cos(|t\xi|)\hat{f}(\xi)d\xi+ \int_{\mathbb{R}^d}e^{ix\xi}\sin(|t\xi|)\frac{\hat{g}(\xi)}{|\xi|}d\xi.
\end{equation*}
\begin{theorem}\label{thmloc}
Let $u,f,g$, $\kappa(p)$ be as above and assume that $2<p<\infty$. Then for all $\kappa<\kappa(p)$, we have
\begin{equation}\label{loc1}
  \|u\|_{L_p(L_\infty([1,2]\times\mathbb{R}^2)\overline{\otimes}\mathcal{M})}\leq C_{p,\kappa}\big(\|f\|_{L_{p,s_p-\kappa}(L_\infty(\mathbb{R}^2)\overline{\otimes}\mathcal{M})}+\|g\|_{L_{p,s_p-\kappa-1}(L_\infty(\mathbb{R}^2)\overline{\otimes}\mathcal{M})}\big),
\end{equation}
where  $\|f\|_{L_{p,s}(L_\infty(\mathbb{R}^2)\overline{\otimes}\mathcal{M})}=\big\|\big((1+|\cdot|)^{s/2}\hat{f}\big)^\vee\big\|_{L_{p}(L_\infty(\mathbb{R}^2)\overline{\otimes}\mathcal{M})}$  , and $C_{p,\kappa}$ is a constant that depends on $p,\kappa$.
\end{theorem}

As in the classical setting, Theorem \ref{qeloc} will be regarded as boundedness of a type of noncommutative Fourier integral operator $T$. Applying the Littlewood-Paley decomposition,  $T$ can be expressed as sum of some localization operators $\{T_j\}_{j\in\mathbb{N}}$, which serves to reduce the proof to the boundedness of each $T_j$. For small $j$, the boundedness follows from the  kernel estimate of $T_j$ and the noncommutative  Hardy-Littlewood maximal inequality. For large $j$, combined with the noncommutative Littlewood-Paley inequality and some properties of noncommutative maximal norm,  the boundedness of $T_j$ reduces to the following:
 \begin{equation}\label{IntTh3}
  \|F_jf\|_{L_p\left(L_\infty(\mathbb{R}^3)\overline{\otimes}\mathcal{M}\right)}\leq C_{\mu}2^{\mu j}\|f\|_{L_p\left(L_\infty(\mathbb{R}^2)\overline{\otimes}\mathcal{M}\right)},
   \end{equation}
where $F_jf(x,t)$ is an operator-valued Fourier integral operator and we refer to Proposition \ref{theorem3} for the precise definition.

  To get the estimate \eqref{IntTh3}, we need numerous modifications and improvements based on the main idea of \cite{MSS} to overcome the difficulties due to noncommutativity. This will constitute the main technical parts of the paper. Roughly speaking, a key fact in our proof is that $|x^*x|^2$ is not necessarily equal to $|xx|^2$ for a general operator $x$.
Thus there are some new inequalities such as {\it noncommutative square functions associated to uniform decomposition and noncommutative Kakeya type maximal inequalities} as well as new geometric structures that need to be handled as in \cite{Lai21, HLW23}. Especially in this paper, we need to prove a new geometric estimate relate to some subsets in $\mathbb{R}^3$ which is much more complicated than the two dimensional  geometric structures in \cite{Lai21, HLW23}. Notably, our alternative approach enables a slight refinement of the geometric estimate originally derived by Muckenhoupt, Seeger, and Sogge in \cite{MSS}. We refer to Remark \ref{remark4} and Remark \ref{rem5} for details.

 As in classical case, it is the estimate \eqref{IntTh3} that will provide one possible way  to the noncommutative version of Bourgain's circular maximal inequality.
Recall that the circular means with radius $t>0$ is defined by
\begin{equation*}
  \mathcal{A}_tf(x):=\int_{S^1}f(x-ty)d\mu(y),\quad f\in\mathcal{S}(\mathbb{R}^2)\otimes \mathcal S(\mathcal M),
\end{equation*}
where $d\mu$ is the induced Lebesgue measure on the circle $S^1$.  Stein's original approach \cite{St76, St78} is to embed $\mathcal{A}_t$ into an analytic family of linear operators $\mathcal{A}_t^\sigma$ of complex order $\sigma$, which is defined as
\begin{equation}\label{cir}
  \mathcal{A}_t^\sigma f(x):=\frac{1}{\Gamma(\sigma)}\int_{\mathbb{R}^2}(1-|y/t|^2)_+^{\sigma-1}f(x-y)dy, \quad \sigma\in\mathbb{C}.
\end{equation}
Note that $\mathcal{A}_t^\sigma$ is a priori defined for $\Re \sigma>0$ and extends to all $\sigma\in\mathbb{C}$ by analytic continuation. Then one may conclude $\mathcal{A}_t=\mathcal{A}_t^0$ by the Fourier transform since $\mathcal{A}_t^\sigma f(x)=(m_{\sigma}(t|\cdot|)\hat{f})^{\vee}(x)$,
where
\begin{equation*}
 m_{\sigma}(r)=2^{\sigma}\pi r^{-\sigma}J_\sigma(r),\quad r>0,
\end{equation*}
with $J_\sigma$ being the Bessel function (see e.g. \cite[Theorem 7]{St78}).

As in \cite{MSS}, we prove the maximal inequality associated to circular means $(\mathcal{A}_t^\sigma)_{t>0}$ with $\Re \sigma>-\kappa(p)$ and thus resolve the conjecture on the noncommutative analogue of Bourgain's circular maximal inequality in Section \ref{Se3.1}.

 \begin{theorem}\label{theorem1}
   Suppose that $2<p<\infty$, then for all $f\in {L_p(L_\infty(\mathbb{R}^2)\overline{\otimes}\mathcal{M})}$,
     \begin{equation}\label{Th1}
       \Big\|{\sup_{0<t<\infty}}^{+} \mathcal{A}_t^\sigma f\Big\|_{L_p(L_\infty(\mathbb{R}^2)\overline{\otimes}\mathcal{M})}\leq C_{p,\sigma}\|f\|_{L_p(L_\infty(\mathbb{R}^2)\overline{\otimes}\mathcal{M})},\qquad if\ \ \Re \sigma>-\kappa(p),
     \end{equation}
   where  $\mathcal{A}_t^\sigma f(x)=(m_{\sigma}(t|\cdot|)\hat{f})^{\vee}(x)$ and $C_{p,\sigma}$ is a constant that depends on $p,\sigma$.
 \end{theorem}

  \begin{remark}
{\rm  The argument for the implication from the estimate \eqref{IntTh3} to Theorem \ref{theorem1} is more complicated than the one to Theorem \ref{thmloc}, since it involves the noncommutative maximal norm and a completely bounded version of Sobolev embedding inequality in the category of operator space, which are well-known to be two of the most difficult objects in noncommutative analysis. We refer the reader to Subsection \ref{Subs3.2} and Subsection \ref{Se3.2} for details.}

 \end{remark}

This paper is organised as follows:
\begin{itemize}
  \item In Section \ref{S2}, we  present definitions, notions and notation mentioned above.
  \item In Section \ref{Se3}, we show how the estimate \eqref{IntTh3} (i.e. Proposition \ref{theorem3} in Section \ref{Subs3.2}) implies Theorem \ref{thmloc} and Theorem \ref{theorem1}. Furthermore, we reduce the proof of Proposition \ref{theorem3} to the case $p=4$.
  \item In Section \ref{Se4} and Section \ref{Se5}, we complete the proof of Proposition \ref{theorem3}.
  \item In Section \ref{S6}, we discuss the transference technique and give the proof of Theorem \ref{qeloc}.
  \item Finally, in Appendix \ref{AP1}, we discuss the wave equation on $\mathcal{R}_\theta^d$ and prove the fixed time $L_p$ estimate.
\end{itemize}
  \textbf{Notation:}  In what follows, we write
$A\lesssim_\alpha B$ if $A\le C_\alpha B$ for some constant $C_\alpha>0$ only depending on the index $\alpha$ or we may just write $A\lesssim B$, and we write $A\thickapprox B$ to mean that $A\lesssim B$ and $B \lesssim A$.
For a function $f$, we set $\tilde{f}(\cdot):=f(-\cdot)$ and $\hat{f}$ (resp. $\check{f}$) denotes the Fourier transform (resp. the inverse Fourier transform) of $f$. The notation $T_m$ stands for the Fourier multiplier associated with the symbol $m$. Given a function $f$, $\Re f$ denotes the real part of $f$  while  $\Im f$ represents the imaginary part. In this paper, we denote the tensor von Neumann algebra $(L_\infty(\mathbb{R}^3)\overline{\otimes}\mathcal{M},\int_{\mathbb{R}^3}d\mu\overline{\otimes}\tau)$ by $(\mathcal{N}, \varphi)$ and use these notions freely.

\section{Preliminaries}\label{S2}
\subsection{Noncommutative $L_p$ spaces}
 Let $(\mathcal{M},\tau)$ be a semifinite von Neumann algebra $\mathcal{M}$ endowed with  a normal semifinite faithful trace ($n.s.f.$ in short) $\tau$.
 Let $\mathcal{S}_+(\mathcal{M})$ be the set of all $x\in\mathcal{M}_+$ such that $\tau(\mathrm{supp}(x))<\infty$, where $\mathrm{supp}(x)$ is the support of $x$ (i.e. the least projection $e\in\mathcal{M}$ such that $ex=x=xe$)  and $\mathcal{S}(\mathcal{M})$ be the linear span of $\mathcal{S}_+(\mathcal{M})$. Then $\mathcal{S}(\mathcal{M})$ is a weak* dense $*$-subalgebra of $\mathcal{M}$. For $1\leq p<\infty$, we define \beqq \|x\|_p=(\tau(|x|^p))^{1/p},\quad x\in\mathcal{S}(\mathcal{M}),\eeqq
where $|x|=(x^*x)^{1/2}$ represents the modulus of $x$. The quantity $\|\cdot\|_p$ is a norm, and thus $(\mathcal{S}(\mathcal{M}),\|\cdot\|_p)$ forms a normed vector space. The completion of this space is known as the noncommutative $L_p$ space associated with $(\mathcal{M},\tau)$, denoted by  $L_p(\mathcal{M})$.
In the case $p=\infty$, we set $L_\infty(\mathcal{M})=\mathcal{M}$, and  define $\|x\|_\infty:=\|x\|_\mathcal{M}$.
 As the commutative $L_p$ spaces, the noncommutative $L_p$ spaces enjoy the basic properties such as the H\"{o}lder inequality, the duality, the interpolation etc.. For more information on the noncommutative $L_p$ spaces, we refer to \cite{PX03}.

 In this paper, we focus on the noncommutative $L_p$ spaces  associated with the pair $(L_\infty(\Sigma, \mu)\overline{\otimes}\mathcal{M},\int d\mu\overline{\otimes}\tau)$, where $(\Sigma,\mu)$ is a measurable space. Note that the Bochner $L_p$ space $L_p(\Sigma;L_p(\mathcal{M}))$ is isomorphic to $L_p(L_\infty(\Sigma, \mu)\overline{\otimes}\mathcal{M})$.  We will frequently employ the two well-known convexity inequalities involving operator-valued functions (see e.g. \cite[Lemma 2.4]{HLSX}, \cite[Page 9]{Mei09}).

 \begin{lemma}\label{convex}
 Let $(\Sigma,\mu)$ be a  measurable space. Suppose that $f$ is a $\mathcal{M}$-valued function on $\Sigma$ which is weak*-integrable and $g$ is a complex-valued integrable function. Then we have the following operator inequality,
 \begin{equation*}
   \Big|\int_{\Sigma}f(x)g(x)d\mu(x)\Big|^2\leq\int_{\Sigma}|g(x)|^2d\mu(x)\times\int_{\Sigma}|f(x)|^2d\mu(x).
 \end{equation*}
 \end{lemma}

 \begin{lemma}\label{convex1}
 Let $(\Sigma,\mu)$ be a  measurable space. Suppose that $f$ is a weak*-integrable  $\mathcal{M}_+$-valued function on $\Sigma$ and $g$ is a positive integrable function. Then for $1<p<\infty$, we have the following operator inequality,
 \begin{equation*}
  \int_{\Sigma}f(x)g(x)d\mu(x)\leq\Big(\int_{\Sigma}|g(x)|^{p^\prime}d\mu(x)\Big)^{1/{p^\prime}}\times\Big(\int_{\Sigma}|f(x)|^pd\mu(x)\Big)^{1/p}.
 \end{equation*}
 Here, $p^\prime$ denotes the conjugate index of $p$, i.e. $\frac{1}{p}+\frac{1}{p^\prime}=1$.
 \end{lemma}

The column and row function spaces $L_p(\mathcal{M};\ell_2^c)$ and $L_p(\mathcal{M};\ell_2^r)$ play an important role in noncommutative analysis, as they provide a framework for defining the noncommutative square function. Here, we expand on the definitions and some of their key properties.
For $1\leq p\leq\infty$, given a finite sequence $(x_n)$ in $L_p(\mathcal{M})$. Define
 \begin{equation*}
   \|(x_n)\|_{L_p(\mathcal{M};\ell_2^c)}:=\Big\|\big(\sum_n|x_n|^2\big)^{1/2}\Big\|_{L_p(\mathcal{M})},\quad  \|(x_n)\|_{L_p(\mathcal{M};\ell_2^r)}:=\|(x_n^*)\|_{L_p(\mathcal{M};\ell_2^c)}.
 \end{equation*}
 For $1\leq p<\infty$, $L_p(\mathcal{M};\ell_2^c)$ (resp. $L_p(\mathcal{M};\ell_2^r ) )$ are defined as the completions of the set of all finite sequences in $L_p(\mathcal{M})$ with respect to $\|\cdot\|_{L_p(\mathcal{M};\ell_2^c)}$ (resp. $\|\cdot\|_{L_p(\mathcal{M};\ell_2^r)} )$. For $p=\infty$, the Banach space $L_\infty(\mathcal{M};\ell_2^c)$ (resp. $L_\infty(\mathcal{M};\ell_2^r ))$ consists of all sequences in $L_\infty(\mathcal{M})$ such that $\sum_n x_n^*x_n$ (resp. $\sum_n x_nx_n^*$) converges in the weak*-topology.
 Then we introduce some properties related to the column and row function spaces, such as the H\"{o}lder inequality, complex interpolation.

 \begin{lemma}[\cite{PX03}]\label{sq}
Given $0<p,q,r\leq\infty$ satisfying the relation $\frac{1}{r}=\frac{1}{p}+\frac{1}{q}$. Then for any sequences $(x_n)\in L_p(\mathcal{M};\ell_2^c)$ and $(y_n)\in L_q(\mathcal{M};\ell_2^c)$, we have
\begin{equation*}
  \Big\|\sum_nx_n^*y_n\Big\|_{L_r(\mathcal{M})}\leq \|(x_n)\|_{L_p(\mathcal{M};\ell_2^c)}\|(y_n)\|_{L_q(\mathcal{M};\ell_2^c)}.
\end{equation*}
\end{lemma}

\begin{lemma}[\cite{PX03}]\label{sq1}
Given $1\leq p_0,p_1\leq\infty$ and $0<\alpha<1$. Let $\frac{1}{p}=\frac{1-\alpha}{p_0}+\frac{\alpha}{p_1}$, then we have isometrically
\begin{equation*}
  \big(L_{p_0}(\mathcal{M};\ell_2^c),L_{p_1}(\mathcal{M};\ell_2^c)\big)_\alpha=L_p(\mathcal{M};\ell_2^c).
\end{equation*}
This equality also holds for the row function spaces.
\end{lemma}
Based on Lemma \ref{sq1}, one can deduce the following lemma:
\begin{lemma}\label{emb}
For $2\leq p\leq\infty$ and any sequence $(x_n)\in L_p(\mathcal{M};\ell_2^c)$, it holds that
\begin{equation*}
 \Big(\sum_{n}\|x_n\|_p^p\Big)^{\frac{1}{p}}\leq \Big\|\Big(\sum_{n}|x_n|^2\Big)^{\frac{1}{2}}\Big\|_p.
\end{equation*}
\end{lemma}

Finally, we present the Littlewood-Paley inequality, as established in \cite{MP}. Let $\big(\Delta_{\ell}\big)_{\ell\in\mathbb{Z}}$ be the  Littlewood-Paley operator (see the proof of Lemma \ref{leA_j} for specific definition), the following Littlewood-Paley inequality holds.

\begin{lemma}\label{LP}
 Suppose that $\big(\Delta_{\ell}\big)_{\ell\in\mathbb{Z}}$ is the Littlewood-Paley operator and $2\leq p<\infty$. Then for any $f\in L_p(L_\infty(\mathbb R^2)\overline{\otimes}\mathcal{M})$,
\begin{equation*}
  \max\Big\{\big\|\big(\Delta_{\ell}f\big)_{\ell\in\mathbb{Z}}\big\|_{L_p(\ell_2^c)}, \big\|\big(\Delta_{\ell}f\big)_{\ell\in\mathbb{Z}}\big\|_{L_p(\ell_2^r)}\Big\}\lesssim \|f\|_{p}.
\end{equation*}
\end{lemma}

\subsection{Noncommutative $\ell_\infty$-valued $L_p$ spaces}
In the noncommutative setting, the maximal norm requires an equivalent definition since $\sup_n|x_n|$ does not make sense for a sequence $(x_n)_n$ of operators. We adopt the definition of the noncommutative maximal norm which was introduced by Pisier \cite{Pis98} and later generalized by Junge \cite{Jun02}.

Given $1\leq p\leq\infty$, the Banach space $L_p(\mathcal{M};\ell_\infty)$ is defined as the space of all sequences $x=(x_n)_{n\in\mathbb{N}}$ in $L_p(\mathcal{M})$ that  admit a factorization through elements $a,b\in L_{2p}(\mathcal{M})$ and a bounded sequence $y=(y_n)_{n\in\mathbb{N}}$ in $L_\infty(\mathcal{M})$ such that $$x_n=ay_nb,\quad \forall n\in\mathbb{N}.$$ The norm of $x$ in $L_p(\mathcal{M};\ell_\infty)$ is given by
\begin{equation*}
  \|x\|_{L_p(\mathcal{M};\ell_\infty)}:=\inf\Big\{\|a\|_{2p}\sup_{n\in\mathbb{N}}\|y_n\|_\infty\|b\|_{2p}\Big\},
\end{equation*}
where the infimum runs over all factorizations of $x$ as above. Additionally, it is worth noting that for a sequence $x=(x_n)_{n\in\mathbb{N}}$ of self-adjoint operators in $L_p(\mathcal{M})$, $x\in L_p(\mathcal{M};\ell_\infty)$ if and only if there exists a positive element $a\in L_p(\mathcal{M})$ such that
\begin{equation*}
  -a\leq x_n\leq a\quad\text{for all}\ n\in\mathbb{N}.
\end{equation*}
In this case, we have
\begin{equation}\label{maxnorm}
  \|x\|_{L_p(\mathcal{M};\ell_\infty)}=\inf\{\|a\|_p\ :\ a\in L_p(\mathcal{M}),\ -a\leq x_n\leq a,\ \forall n\in\mathbb{N}\}.
\end{equation}
For this reason, the maximal norm is intuitively denoted by
$\|\sup^+_nx_n\|_p$, but it is crucial to emphasize that $\|\sup^+_nx_n\|_p$ merely is a notation and $\sup^+_nx_n$ does not make sense in the noncommutative setting. In fact, we find that the notation $\|\sup^+_nx_n\|_p$ offers a more intuitive understanding compared to the formal expression $\|x\|_{L_p(\mathcal{M};\ell_\infty)}$.
More broadly, this concept extends to any index set $\Lambda$. Specifically, given any index set $\Lambda$, $L_p(\mathcal{M};\ell_\infty(\Lambda))$ denotes the space of all $x=(x_\lambda)_{\lambda\in\Lambda}$ in $L_p(\mathcal{M})$  admitting a factorization of the form $x_\lambda=ay_\lambda b,$ where $a,b\in L_{2p}(\mathcal{M}), y_\lambda\in L_\infty(\mathcal{M})$ and $\sup_{\lambda\in\Lambda}\|y_\lambda\|_\infty<\infty.$
The norm in $L_p(\mathcal{M};\ell_\infty(\Lambda))$ is then defined as
\begin{equation*}
  \|{\sup_{\lambda\in\Lambda}}^+ x_\lambda\|_p:=\inf_{x_\lambda=ay_\lambda b}\Big\{\|a\|_{2p}\sup_{\lambda\in\Lambda}\|y_\lambda\|_\infty\|b\|_{2p}\Big\}.
\end{equation*}
As demonstrated in \cite{JX07}, $x\in L_p(\mathcal{M};\ell_\infty(\Lambda))$ if and only if
\begin{equation*}
  \sup\Big\{\|{\sup_{\lambda\in I}}^+ x_\lambda\|_p\ :\  I\ \text{ is a finite subset of}\ \Lambda\Big\}<\infty.
\end{equation*}
Indeed, this supremum coincides with $\|{\sup_{\lambda\in\Lambda}}^+ x_\lambda\|_p$.
For the sake of simplicity and clarity, when no confusion arises, we may continue to denote the space $L_p(\mathcal{M};\ell_\infty(\Lambda))$ simply as $L_p(\mathcal{M};\ell_\infty)$.

On the other hand, we can define a closely related Banach space $L_p(\mathcal{M};\ell_\infty^c)$ for $p\geq2$, which consists of all sequences
$x=(x_\lambda)_{\lambda\in\Lambda}$ in $L_p(\mathcal{M})$ such that
\begin{equation*}
  \big\|{\sup_{\lambda\in\Lambda}}^+ |x_\lambda|^2\big\|_{p/{2}}^{1/2}<\infty,
\end{equation*}
with the norm defined as $\|x\|_{L_p(\mathcal{M};\ell_\infty^c)}:=\|{\sup_{\lambda\in\Lambda}}^+ |x_\lambda|^2\|_{p/{2}}^{1/2}$. Furthermore, the Banach space
$L_p(\mathcal{M};\ell_\infty^r)$ for $p\geq2$ is defined as
 \begin{equation*}
   L_p(\mathcal{M};\ell_\infty^r):=\big\{x=(x_\lambda)_{\lambda\in\Lambda}\ :\ x^*=(x_\lambda^*)_{\lambda\in\Lambda}\in L_p(\mathcal{M};\ell_\infty^c)\big\}
 \end{equation*}
 with the norm $\|x\|_{L_p(\mathcal{M};\ell_\infty^r)}:=\|{\sup_{\lambda\in\Lambda}}^+ |x_\lambda^*|^2\|_{p/{2}}^{1/2}$.
 The following useful interpolation results can be found in \cite{HWW}, \cite{JP10} and \cite{JX07}.
\begin{lemma}\label{max interpolation}
{\rm (i)}  Suppose $1\leq p_0,p_1\leq\infty$ , $0<\alpha<1$. Let $\frac{1}{p}=\frac{1-\alpha}{p_0}+\frac{\alpha}{p_1}$, then we have isometrically:
      \begin{equation*}
        L_p(\mathcal{M};\ell_\infty)=\big(L_{p_0}(\mathcal{M};\ell_\infty), L_{p_1}(\mathcal{M};\ell_\infty)\big)_\alpha.
      \end{equation*}

{\rm (ii)} For $2\leq p\leq\infty$, we have
  \begin{equation*}
    L_p(\mathcal{M};\ell_\infty)=\big(L_{p}(\mathcal{M};\ell_\infty^c), L_{p}(\mathcal{M};\ell_\infty^r)\big)_{1/2}.
  \end{equation*}
\end{lemma}

\subsection{Quantum Euclidean spaces }
 Let $\theta$ be a $d\times d $ real  antisymmetric matrix and for each $t\in\mathbb{R}^d$, we define the unitary
operator $U_\theta(t)$ acting on $L_2(\mathbb{R}^d)$ as follows:
\beq\label{wll}
		(U_\theta(t)f)(r):=e^{-\frac i2(t,\theta r)}f(r-t),\quad f\in{L_2(\mathbb{R}^d)}, r\in\mathbb{R}^d.
	\eeq
The family $\{U_\theta(t)\}_{t\in\mathbb{R}^d}$ is strongly continuous and satisfies the Weyl relation $$U_\theta(t)U_\theta(s)=e^{\frac i2(s,\theta t)}U_\theta(t+s),\quad U_{\theta}^*(t)=U_\theta(-t)$$ for $s,t\in\mathbb{R}^d$.
The von Neumann subalgebra of $\mathcal{B}(L_2(\mathbb{R}^d))$ generated by $\{U_\theta(t)\}_{t\in\mathbb{R}^d}$, denoted $\mathcal{R}_\theta^d$, is referred to as quantum Euclidean space.
In the special case when $\theta=0$, $\mathcal{R}_\theta^d $ reduces to the von Neumann algebra generated by the unitary group of translations on $\mathbb{R}^d$, which is
  $*$-isomorphic to $L_\infty(\mathbb{R}^d)$. We recommend that readers consult, for example, the references \cite{GJP, HLW23, LSZ, Xiong1} for further information on $\mathcal{R}_\theta^d$.

As outlined in \cite{Xiong1}, an injective map, also denoted by $U_\theta$, can be defined from $L_1(\mathbb{R}^d)$ into the quantum Euclidean space $\mathcal{R}_\theta^d $:
\begin{equation}\label{defU}
 U_\theta(f):=\int_{\mathbb{R}^d}f(t)U_\theta(t)dt,\quad f \in L_1(\mathbb{R}^d).
\end{equation}
Furthermore, the class of Schwartz functions within $\mathcal{R}_\theta^d $ is defined as :
$$ \mathcal{S}(\mathcal{R}_\theta^d):=\{x\in\mathcal{R}_\theta^d:\quad x=U_\theta(f), \text{\:for\: } f\in\mathcal{S}(\mathbb{R}^d)\}.$$
One can see $U_\theta$ is a bijection between $\mathcal{S}(\mathbb{R}^d)$ and $ \mathcal{S}(\mathcal{R}_\theta^d)$,
and thus $\mathcal{S}(\mathcal{R}_\theta^d)$ is a Fr\'{e}chet topological space equipped with the Fr\'{e}chet topology induced by $U_\theta$. Moreover, we denote the space of continuous linear functionals on $\mathcal{S}(\mathcal{R}_\theta^d)$ as $\mathcal{S}^\prime(\mathcal{R}_\theta^d)$, and $U_\theta$ extends to a bijection between $\mathcal{S}^\prime(\mathbb{R}^d)$ and $\mathcal{S}^\prime(\mathcal{R}_\theta^d)$, where for $f\in\mathcal{S}^\prime(\mathbb{R}^d)$,
\beq (U_\theta(f),U_\theta(g)):=(f, \tilde{g}), \quad \text{for\:all\;}g\in \mathcal{S}(\mathbb{R}^d).\eeq
Given $x\in \mathcal{S}(\mathcal{R}_\theta^d)$ represented as $x=U_\theta(f)$ for some $f\in \mathcal{S}(\mathbb{R}^d)$, we define $\tau_\theta(x):=f(0)$, then $\tau_\theta$ extends to a $n.s.f.$ trace on $\mathcal{R}_\theta^d$. The associated noncommutative $L_p$ space is denoted $L_p(\mathcal{R}_\theta^d)$.
Additionally, the space $ \mathcal{S}(\mathcal{R}_\theta^d)$ is dense in $L_p(\mathcal{R}_\theta^d)$ for $1\leq p<\infty$ with respect to the norm, and is also dense in $L_\infty(\mathcal{R}_\theta^d)$ in the weak* topology. See \cite{GJP, Xiong1} for more information.

\begin{definition}
  Suppose that $m$ be an essentially bounded function on $\mathbb{R}^d$. Let $T_m$ denote the corresponding Fourier multiplier, defined as
  \begin{equation*}
    T_m(U_\theta(f)):=U_\theta(mf), \quad \forall\ U_\theta(f)\in\mathcal{S}(\mathcal{R}_\theta^d).
  \end{equation*}
\end{definition}
Now we introduce the $L_p$-Sobolev space on $\mathcal{R}_\theta^d$, one can refer to \cite{M23, Xiong1}.
\begin{definition}
 For $1\leq p\leq\infty$ and $s\in\mathbb{R}$, let $J_s(\xi)=(1+|\xi|^2)^{s/2}$ be a function on $\mathbb{R}^d$. The  $L_p$-Sobolev space $L_{p,s}(\mathcal{R}_\theta^d)$ is defined as the subset of $\mathcal{S}^\prime(\mathcal{R}_\theta^d)$ consisting of elements $x$  such that $T_{J_s}(x)\in L_p(\mathcal{R}_\theta^d)$, with the norm given by:
 \begin{equation*}
   \|x\|_{L_{p,s}(\mathcal{R}_\theta^d)}:=\|T_{J_s}(x)\|_{L_p(\mathcal{R}_\theta^d)}.
 \end{equation*}

\end{definition}
\begin{remark}
 \rm{The $L_p$-Sobolev space defined here is inhomogeneous. One can similarly define the homogeneous Sobolev space on quantum Euclidean space. In general, the differential operators can be defined by the Fourier multiplier. For example, the Laplace operator  $\Delta_\theta$ on $\mathcal{R}_\theta^d$ is defined with the Fourier symbol $m(\xi)=-|\xi|^2$, $\xi\in\mathbb{R}^d$.}
\end{remark}

\begin{lemma}[see  \cite{HLW23} or \cite{Xiong1}]\label{HY}
For any $f\in \mathcal{S}(\mathbb{R}^d)$, we have
\beq \label{hy1}
\|U_\theta(f)\|_{L_{2}(\mathcal{R}_\theta^d)}=\|f\|_{L_2(\mathbb{R}^d)},\eeq
and  for $1\leq p<2$, \beq \label{hy2}
\|U_\theta(f)\|_{L_{p^\prime}(\mathcal{R}_\theta^d)}\leq\|f\|_{L_{p}(\mathbb{R}^d)}.\eeq
Thus $U_\theta$ extends to a contraction from $L_p(\mathbb{R}^d)$ ($1\leq p\leq2$) to $L_{p^\prime}(\mathcal{R}_\theta^d)$, with the extension being an isometry when $p=2$.
\end{lemma}

\section{Proof of the main Theorems}\label{Se3}
Recall that the circular mean operator $\mathcal{A}_t^\sigma$ for a given $t>0$ is a Fourier multiplier defined as:
\begin{equation*}
 \mathcal{A}_t^\sigma f(x)=\int_{\mathbb{R}^2}e^{ix\xi}m_\sigma(t|\xi|)\hat{f}(\xi)d\xi,
\end{equation*}
where $m_\sigma(r)=2 ^\sigma \pi r^{-\sigma}J_{\sigma}(r)$, with  $J_\sigma$ being the Bessel function given by
\begin{equation*}
J_\sigma(r)=\frac{(\frac{r}{2})^\sigma}{\Gamma(\sigma+\frac12)\Gamma(\frac12)}
\int_{-1}^1e^{irs}(1-s^2)^{\sigma-\frac12}ds.
\end{equation*}
 More information for Bessel function we refer to \cite[Appendix B]{GF3} and \cite[Page 338]{St93}. For the sake of completeness, we present here the asymptotic expansion of $J_{\sigma}$ under the condition $\Re\sigma>-1/2$.
 \begin{lemma}\label{Bessel}
 Let $J_{\sigma}$ be the Bessel function with $\Re\sigma>-1/2$, then for $r\geq1$, one has
 \begin{equation*}
  J_{\sigma}(r)=\Big(\frac{2}{\pi r}\Big)^{1/2}\cos\Big(r-\frac{\pi\sigma}{2}-\frac{\pi}{4}\Big)+O(r^{-3/2}).
 \end{equation*}
 More generally, $J_{\sigma}$ admits a complete asymptotic expansion: for any $N\geq1$,
 \begin{equation}\label{ae}
  J_{\sigma}(r)=r^{-1/2}e^{ir}\sum_{j=0}^Na_jr^{-j}+ r^{-1/2}e^{-ir}\sum_{j=0}^Nb_jr^{-j}+ R_{\sigma,N}(r),
 \end{equation}
 where $a_j, b_j$ are suitable constants and the error term $R_{\sigma,N}$ satisfies
 \begin{equation*}
  \Big| \Big(\frac{d }{dr}\Big)^kR_{\sigma,N}(r)\Big|\leq C_kr^{-N-k}\quad \text{for}\ k\in\mathbb{N},\ r\geq1.
 \end{equation*}
 In particular, one has $J_{1/2}(r)=cr^{-1/2}\sin(r)$ and $J_{-1/2}(r)=c^\prime r^{-1/2}\cos(r)$.
 \end{lemma}

\subsection{Proof of Theorem \ref{theorem1}}\label{Se3.1}
In this subsection, we present the proof of Theorem \ref{theorem1}, which fundamentally relies on the Littlewood-Paley decomposition.
Let $\beta\in C_0^\infty(\mathbb{R})$ be a nonnegative function with support in $[1/2,2]$ such that $\sum_{j\in\mathbb{Z}}\beta(2^{-j}r)=1$ for all $r>0$. Set $\eta(r):=1-\sum_{j\geq0}\beta(2^{-j}r)$, and
\begin{equation}\label{defA_0}
  \mathcal{A}_{t,0}^\sigma f(x):=\int_{\mathbb{R}^2}e^{ix\xi}m_\sigma(t|\xi|)\eta(|t\xi|)\hat{f}(\xi)\ d\xi,
\end{equation}
and for $j\geq1$
\begin{equation}\label{defA_j}
  \mathcal{A}_{t,j}^\sigma f(x):=\int_{\mathbb{R}^2}e^{ix\xi}m_\sigma(t|\xi|)\beta(|2^{-j}t\xi|)\hat{f}(\xi)\ d\xi.
\end{equation}
Then $\mathcal{A}_t^\sigma f(x)=\sum_{j=0}^\infty \mathcal{A}_{t,j}^\sigma f(x)$ and
$$\left\|{\sup_{t>0}}^+ \mathcal{A}_{t}^\sigma f\right\|_p\leq\sum_{j=0}^\infty \left\|{\sup_{t>0}}^+ \mathcal{A}_{t,j}^\sigma f\right\|_p,$$
which implies Theorem \ref{theorem1} once we have Lemmas
\ref{leA_0} and \ref{leA_j} below.

\begin{lemma}\label{leA_0}
Let $p>1$, for any $f\in L_p(L_\infty(\mathbb{R}^2)\overline{\otimes}\mathcal{M})$, one has $$ \Big\|{\sup_{t>0}}^+ \mathcal{A}_{t,0}^\sigma f\Big\|_p\lesssim \|f\|_p.$$

\end{lemma}
\begin{proof}
 Without loss of generality, we may assume that $f$ is positive. Recall that
  \begin{align*}
  \mathcal{A}_{t,0}^\sigma f(x)&=\int_{\mathbb{R}^2}e^{ix\xi}m_\sigma(t|\xi|)\eta(|t\xi|)\hat{f}(\xi)\ d\xi       \\
&=\int_{\mathbb{R}^2}K_{t,0}(x-y)f(y)\ dy,
  \end{align*}
where $K_{t,0}(x)=\int_{\mathbb{R}^2}e^{ix\xi}m_\sigma(t|\xi|)\eta(|t\xi|)\ d\xi$. Using integration by parts, we get
\begin{equation} \label{esk1}
 |K_{t,0}(x)|\leq C_N t^{-2}(1+|x|/t)^{-N}
\end{equation}
for any integer $N$.
Since $K_{t,0}$ can be written as a linear combination of four nonnegative functions which still satisfy the above estimate \eqref{esk1},  for simplicity, we may directly assume $K_{t,0}$ to be nonnegative. Then,
 \begin{align*}
  \mathcal{A}_{t,0}^\sigma f(x)&=\int_{\mathbb{R}^2}K_{t,0}(x-y)f(y)\ dy\\
  &\leq C_N\left(\int_{\frac{|x-y|}{t}<1}f(y)\ dy+\sum_{k=1}^\infty \int_{2^{k-1}\leq\frac{|x-y|}{t}< 2^k}t^{-2}(1+|x-y|/t)^{-N}f(y)\ dy\right)\\
  &\leq \sum_{k=0}^\infty C_N2^{-kN}t^{-2}\int_{B(x,2^kt)}f(y)\ dy= \sum_{k=0}^\infty \frac{C_N2^{-k(N-2)}}{|B(x,2^kt)|}\int_{B(x,2^kt)}f(y)\ dy.
  \end{align*}
By choosing $N>2$ and applying the noncommutative Hardy-Littlewood maximal inequality (see \cite[Theorem 3.3 ]{Mei09}), one derives the desired estimates.
\end{proof}
  \begin{lemma}\label{leA_j}
Suppose that $2<p<\infty$ and $\Re\sigma> -\kappa(p)$, there exists $\varepsilon(p)>0$ such that the following estimate
\begin{equation}\label{leA_jeq1}
   \Big\|{\sup_{t>0}}^+ \mathcal{A}_{t,j}^\sigma f\Big\|_p\lesssim 2^{-\varepsilon(p)j}\|f\|_p,\quad for\ \ j\geq1,
   \end{equation}
holds for all $f\in {L_p(L_\infty(\mathbb{R}^2)\overline{\otimes}\mathcal{M})}$.
\end{lemma}
Before delving into the proof of Lemma \ref{leA_j}, we require the subsequent proposition, which will be rigorously established in the following subsection.
\begin{proposition}\label{theorem2}
  Suppose that $2<p<\infty$ and $\Re\sigma> -\kappa(p)$, there exists $\varepsilon(p)>0$ such that the following estimate
 \begin{equation}\label{Th2}
  \Big\|{\sup_{1<t<2}}^+ \mathcal{A}_{t,j}^\sigma f\Big\|_p\lesssim2^{-\varepsilon(p)j}\|f\|_p,\quad for\ \ j\geq1,
   \end{equation}
holds for all $f\in {L_p(L_\infty(\mathbb{R}^2)\overline{\otimes}\mathcal{M})}$.
 \end{proposition}
\begin{proof}[\textbf{Proof of Lemma \ref{leA_j}}]
  Let  $\Delta_\ell^j$ be the Littlewood-Paley operator, that is
 \begin{equation*}
    \Delta_\ell^j f(x):=\left(\tilde{\beta}(2^{-j}|2^\ell\cdot|)\hat{f}(\cdot)\right)^\vee(x),
\end{equation*}
where $\tilde{\beta}\in C_0^\infty(\mathbb{R})$ is a nonnegative function which satisfies $\beta(|t\xi|)=\beta(|t\xi|)\tilde{\beta}(|\xi|)$ for any $t\in [1,2)$ and $\xi\in\mathbb{R}^2$. Then,
\begin{align}\label{le21}
  \Big\|{\sup_{t>0}}^+ \mathcal{A}_{t,j}^\sigma f\Big\|_p&=\left\|{\sup_{t>0}}^+\int_{\mathbb{R}^2}e^{ix\xi}m_\sigma(t|\xi|)\beta(|2^{-j}t\xi|)\hat{f}(\xi)\ d\xi\right\|_p\nonumber\\
  &=\left\|{\sup_{\ell\in\mathbb{Z}}\sup_{1\leq t<2}}^+\int_{\mathbb{R}^2}e^{i\xi x}m_\sigma(|2^\ell t\xi|)\beta(|2^{-j+\ell}t\xi|)\tilde{\beta}(|2^{-j+\ell}\xi|)\hat{f}(\xi)\ d\xi\right\|_p\nonumber\\
  &=\left\|{\sup_{\ell\in\mathbb{Z}}\sup_{1\leq t<2}}^+\int_{\mathbb{R}^2}e^{i\xi x}m_\sigma(|2^\ell t\xi|)\beta(|2^{-j+\ell}t\xi|)\widehat{\Delta_\ell^j f}(\xi)\ d\xi\right\|_p\nonumber\\&=\Big\|{\sup_{\ell\in\mathbb{Z}}\sup_{1\leq t<2}}^+\mathcal{A}_{t,j}^\sigma\Big((\Delta_\ell^j f)(2^\ell\cdot)\Big)(2^{-\ell} \cdot)\Big\|_p.\nonumber\\
\end{align}

To proceed with the argument, we assert the following claim:
\begin{equation}\label{le22}
  \Big\|{\sup_{\ell\in\mathbb{Z}}\sup_{1\leq t<2}}^+\mathcal{A}_{t,j}^\sigma\left((\Delta_\ell^j f)(2^\ell\cdot)\right)(2^{-\ell} \cdot)\Big\|_p\lesssim 2^{-\varepsilon(p)j}\Big(\sum_{\ell\in\mathbb{Z}}\|2^{\frac{2\ell}{p}}(\Delta_\ell^j f)(2^\ell\cdot)\|_p^p\Big)^{\frac{1}{p}}.
\end{equation}
We prove the above claim at the end of this proof.
Now combining the inequalities \eqref{le21},  claim \eqref{le22} and Lemma \ref{emb}, we deduce
\begin{align*}
  \Big\|{\sup_{t>0}}^+ \mathcal{A}_{t,j}^\sigma f\Big\|_p &\lesssim 2^{-\varepsilon(p)j}\Big(\sum_{\ell\in\mathbb{Z}}\|2^{\frac{2\ell}{p}}(\Delta_\ell^j f)(2^\ell\cdot)\|_p^p\Big)^{\frac{1}{p}}\\
  &=2^{-\varepsilon(p)j}\Big(\sum_{\ell\in\mathbb{Z}}\|(\Delta_\ell^j f)(\cdot)\|_p^p\Big)^{\frac{1}{p}}\\
  &\leq 2^{-\varepsilon(p)j}\Big\|\Big(\sum_{\ell\in\mathbb{Z}}|\Delta_\ell^j f|^2\Big)^{\frac{1}{2}}\Big\|_p\lesssim 2^{-\varepsilon(p)j}\|f\|_p,
\end{align*}
where the last inequality follows from Lemma \ref{LP} and we complete the proof of Lemma \ref{leA_j} once we prove the claim \eqref{le22}.

For this target, let $f_\ell$ be an element in $\mathcal{S}(\mathbb{R}^2)\otimes\mathcal{M}$ associated with an integer $\ell$, it suffices to show
\begin{equation}\label{le23}
  \Big\|{\sup_{\ell\in\mathbb{Z}}\sup_{1<t<2}}^+\mathcal{A}_{t,j}^\sigma(f_\ell)(2^{-\ell} \cdot)\Big\|_p\lesssim 2^{-\varepsilon(p)j}\Big(\sum_{\ell\in\mathbb{Z}}\|2^{\frac{2\ell}{p}}f_\ell\|_p^p\Big)^{\frac{1}{p}}.
\end{equation}
Fix $\ell\in\mathbb{Z}$, we write $\mathcal{A}_{t,j}^\sigma(f_\ell)$ as $\Re[\mathcal{A}_{t,j}^\sigma(f_\ell)] +i\Im[\mathcal{A}_{t,j}^\sigma(f_\ell)]$ where
\begin{equation*}
\begin{split}
\Re[\mathcal{A}_{t,j}^\sigma(f_\ell)]&=\frac{1}{2} \Big(\mathcal{A}_{t,j}^\sigma(f_\ell)+ \Big(\mathcal{A}_{t,j}^\sigma(f_\ell)\Big)^*\Big),\\
\Im[\mathcal{A}_{t,j}^\sigma(f_\ell)]&=\frac{1}{2i} \Big(\mathcal{A}_{t,j}^\sigma(f_\ell)- \Big(\mathcal{A}_{t,j}^\sigma(f_\ell)\Big)^*\Big).
\end{split}\end{equation*}
By the triangle inequality, a fact that the adjoint map is an isometric
isomorphism on $L_p(L_\infty(\mathbb{R}^2)\overline{\otimes}\mathcal{M};\ell_\infty)$ and applying Proposition \ref{theorem2},
\begin{align*}
  2^{\frac{-2\ell}{p}}\Big\|{\sup_{1<t<2}}^+ \Re[\mathcal{A}_{t,j}^\sigma(f_\ell)](2^{-\ell}\cdot)\Big\|_p&= \Big\|{\sup_{1<t<2}}^+  \Re[\mathcal{A}_{t,j}^\sigma(f_\ell)]\Big\|_p\\
  &\leq \frac{1}{2}\Big\|{\sup_{1<t<2}}^+  \mathcal{A}_{t,j}^\sigma(f_\ell)\Big\|_p + \frac{1}{2}\Big\|{\sup_{1<t<2}}^+  \Big(\mathcal{A}_{t,j}^\sigma(f_\ell)\Big)^*\Big\|_p\\
  &=\Big\|{\sup_{1<t<2}}^+  \mathcal{A}_{t,j}^\sigma(f_\ell)\Big\|_p\lesssim 2^{-\varepsilon(p)j}\|f_\ell\|_p.
  \end{align*}
Thus for $\epsilon>0$, we can find a positive element $F_{\ell,\epsilon}$ such that for all $1<t<2$,
\begin{equation*}
 -F_{\ell,\epsilon}\leq  \Re[\mathcal{A}_{t,j}^\sigma(f_\ell)](2^{-\ell}\cdot)\leq F_{\ell,\epsilon}
\end{equation*}
 and $\|F_{\ell,\epsilon}\|_p^p\lesssim 2^{-p\varepsilon(p)j+2\ell}\|f_\ell\|_p^p+2^{-|\ell|}\epsilon$.
  Let $F_\epsilon:=\Big(\sum_{\ell\in\mathbb{Z}} F_{\ell,\epsilon}^p\Big)^{\frac{1}{p}}$, then
   \begin{equation*}
    -F_\epsilon\leq  \Re[\mathcal{A}_{t,j}^\sigma(f_\ell)](2^{-\ell}\cdot)\leq F_\epsilon
   \end{equation*}
  for all $\ell\in\mathbb{Z}, 1<t<2$. And we have
  \begin{equation*}
  \|F_\epsilon\|_p^p=\sum_{\ell\in\mathbb{Z}}\|F_{\ell,\epsilon}\|_p^p\lesssim 2^{-p\varepsilon(p)j}\Big(\sum_{\ell\in\mathbb{Z}}\|2^{\frac{2\ell}{p}}f_\ell\|_p^p\Big)+\epsilon.
  \end{equation*}
  Let $\epsilon\rightarrow0$, we deduce
  \begin{equation}\label{le211}
   \Big\|{\sup_{\ell\in\mathbb{Z}}\sup_{1<t<2}}^+\Re[\mathcal{A}_{t,j}^\sigma(f_\ell)](2^{-\ell}\cdot)\Big\|_p\lesssim 2^{-\varepsilon(p)j}\Big(\sum_{\ell\in\mathbb{Z}}\|2^{\frac{2\ell}{p}}f_\ell\|_p^p\Big)^{\frac{1}{p}}.
  \end{equation}
 Similarly we also have
  \begin{equation}\label{le212}
   \Big\|{\sup_{\ell\in\mathbb{Z}}\sup_{1<t<2}}^+\Im[\mathcal{A}_{t,j}^\sigma(f_\ell)](2^{-\ell}\cdot)\Big\|_p\lesssim 2^{-\varepsilon(p)j}\Big(\sum_{\ell\in\mathbb{Z}}\|2^{\frac{2\ell}{p}}f_\ell\|_p^p\Big)^{\frac{1}{p}}.
   \end{equation}
   Using the triangle inequality, along with the inequalities \eqref{le211} and \eqref{le212}, we can derive \eqref{le23}.
\end{proof}

\subsection{Proof of Proposition \ref{theorem2}}\label{Subs3.2}
To prove Proposition \ref{theorem2}, we first state a key proposition whose proof will be given later.

 \begin{proposition}\label{theorem3}
 Fix $\rho_0,\rho_1\in C_0^\infty\big((1/8,8)\big)$, let $a(\xi,t)$ be a symbol of order zero, i.e. for any $t\in(1,2), \alpha\in\mathbb{N}^2$,
 \begin{equation*}
   \Big|\left(\frac{\partial}{\partial \xi}\right)^\alpha a(\xi,t)\Big|\leq c_\alpha(1+|\xi|)^{-|\alpha|}.
 \end{equation*}
Define the operator $F_j$ by $$F_jf(x,t):=\rho_1(t)\int_{\mathbb{R}^2}e^{i(x\xi+t|\xi|)}\rho_0(|2^{-j}\xi|)a(\xi,t)\hat{f}(\xi)d\xi$$ for $j\geq1$. Then
 \begin{equation}\label{Th3}
  \|F_jf\|_{L_p\left(L_\infty(\mathbb{R}^3)\overline{\otimes}\mathcal{M}\right)}\leq C_{\mu}2^{\mu j}\|f\|_{L_p\left(L_\infty(\mathbb{R}^2)\overline{\otimes}\mathcal{M}\right)},
   \end{equation}
   for all $\mu>\frac{1}{2}-\frac{1}{p}-\kappa(p)$.
 \end{proposition}

In \cite{MSS}, the authors utilized the Sobolev embedding inequality to prove Proposition \ref{theorem2} in the commutative case. Here, we use the following alternative lemma, which should be regarded as a cognate version of the Sobolev embedding inequality. A more complete theory of the noncommutative Sobolev embedding  inequalities can be found in \cite{Hong+}.
\begin{lemma}\label{remax}
Given $F\in \mathcal{S}(\mathbb{R})\otimes\mathcal{S}(\mathcal{M})$, let $I$ be an interval on $\mathbb{R}$ and $\mathcal{N_I}$ be the tensor von Neumann algebra $L_\infty(I)\overline{\otimes}\mathcal{M}$. Suppose that $2\leq p\leq\infty$ , then for all $0<\lambda\leq |I|$, we have
\begin{align*}
  \Big\|{\sup_{t\in I}}^+F(t)\Big\|_{L_p(\mathcal{M})}\lesssim &\lambda^{-1/p}\Big\|F\Big\|_{L_p(\mathcal{\mathcal{N}_I})}+\lambda^{1-1/p}\Big\|F^\prime\Big\|_{L_p(\mathcal{\mathcal{N}_I})}.
\end{align*}
\end{lemma}
\begin{proof}
 Fix $t\in I$, and let $I_t\subset I$ be an interval containing $t$ with $|I_t|=\lambda$. For any $s\in I_t$,
 \begin{equation*}
   F(t)-F(s)=\int_{s}^t F^\prime(r) dr,
 \end{equation*}
 then integrating both sides over  $s\in I_t$,
 \begin{equation*}
   \lambda F(t)=\int_{I_t}F(s)ds+\int_{I_t}\int_{s}^t F^\prime(r)drds.
 \end{equation*}
Employing Lemma \ref{convex} and Lemma \ref{convex1},  along with the operator monotone property of the function $t^{2/p}$ for $p\geq2$, we derive the following operator inequality,
 \begin{align}\label{opconvex}
   |\lambda F(t)|^2 &\leq2\Big|\int_{I_t}F(s)ds\Big|^2+2\Big| \int_{I_t}\int_{s}^t F^\prime(r)drds\Big|^2\nonumber\\
  &\leq 2\lambda\int_{I_t}|F(s)|^2ds+2\lambda\int_{I_t}|t-s|\int_{s}^t | F^\prime(r)|^2drds\nonumber\\
  &\leq2\lambda^{2-2/p}\Big(\int_{I_t}|F(s)|^pds\Big)^{2/p}+ 2\lambda \int_{I_t}|t-s|^{2-2/p}\Big(\int_{s}^t | F^\prime(r)|^pdr\Big)^{2/p}ds\nonumber\\
  &\leq 2\lambda^{2-2/p}\Big(\int_{I}|F(r)|^pdr\Big)^{2/p} +2\lambda^{4-2/p}\Big(\int_{I}|F^\prime(r)|^pdr\Big)^{2/p}.
 \end{align}
 Note that the complex interpolation theorem for the maximal norm (see Lemma \ref{max interpolation}) implies that
 \begin{equation}\label{e:maximaladj}
  \Big\|{\sup_{t\in I}}^+F(t)\Big\|_{L_p(\mathcal{M})}\leq\Big\|{\sup_{t\in I}}^+|F(t)|^2\Big\|_{L_{p/2}(\mathcal{M})}^{1/4}\Big\|{\sup_{t\in I}}^+|F^*(t)|^2\Big\|_{L_{p/2}(\mathcal{M})}^{1/4}.
 \end{equation}
To prove our lemma,  it suffices to show
 \begin{equation}\label{remax11}
 \Big\|{\sup_{t\in I}}^+|F(t)|^2\Big\|_{L_{p/2}(\mathcal{M})}^{1/2}\lesssim\lambda^{-1/p}\Big\|F\Big\|_{L_p(\mathcal{\mathcal{N}_I})}+\lambda^{1-1/p}\Big\|F^\prime\Big\|_{L_p(\mathcal{\mathcal{N}_I})},
 \end{equation}
since the second term in \eqref{e:maximaladj} can be handled similarly. Proceeding further, by the operator inequality \eqref{opconvex} and the triangle inequality, we deduce
 \begin{align*}
  &\Big\|{\sup_{t\in I}}^+|F(t)|^2\Big\|_{L_{p/2}(\mathcal{M})}\\&\leq 2 \Big\|\lambda^{-2/p}\Big(\int_{I}|F(r)|^pdr\Big)^{2/p} +\lambda^{2-2/p}\Big(\int_{I}|F^\prime(r)|^pdr\Big)^{2/p}\Big\|_{L_{p/2}(\mathcal{\mathcal{M}})}\\
  &\leq 2\lambda^{-2/p}\Big\|\Big(\int_{I}|F(r)|^pdr\Big)^{2/p}\Big\|_{L_{p/2}(\mathcal{\mathcal{M}})} +2\lambda^{2-2/p}\Big\|\Big(\int_{I}|F^\prime(r)|^pdr\Big)^{2/p}\Big\|_{L_{p/2}(\mathcal{\mathcal{M}})}\\
  &=2\lambda^{-2/p}\big\|F\big\|_{L_p(\mathcal{\mathcal{N}_I})}^2+2\lambda^{2-2/p}\big\|F^\prime\big\|_{L_p(\mathcal{\mathcal{N}_I})}^2,
 \end{align*}
which implies \eqref{remax11}.
 \end{proof}

  To the end, we assume that $t\in(1,2)$ and $j$ is a sufficiently large positive integer. We are now ready to prove Proposition \ref{theorem2}.

Recall the definition of the operator $\mathcal{A}_{t,j}^\sigma$:
\begin{equation*}
  \mathcal{A}_{t,j}^\sigma f(x)=\int_{\mathbb{R}^2}e^{ix\xi}m_\sigma(t|\xi|)\beta(|2^{-j}t\xi|)\hat{f}(\xi)d\xi,
\end{equation*}
where (see Lemma \ref{Bessel})
\begin{equation*}
(t|\xi|)^{\sigma} m_{\sigma}(t|\xi|)=(t|\xi|)^{-1/2}e^{it|\xi|}\sum_{\ell=0}^Na_\ell\cdot(t|\xi|)^{-\ell}+ (t|\xi|)^{-1/2}e^{-it|\xi|}\sum_{\ell=0}^Nb_\ell\cdot(t|\xi|)^{-\ell}+R_{\sigma,N}(t|\xi|),
 \end{equation*}
 with $a_\ell,b_\ell$ being suitable constants, and the error term $R_{\sigma,N}$ satisfying the decay condition:
 \begin{equation}\label{decay}
  \Big| \Big(\frac{d }{dr}\Big)^kR_{\sigma,N}(r)\Big|\leq C_kr^{-N-k}\quad \text{for}\ k\in\mathbb{N},\ r\geq1.
 \end{equation}
 We introduce two symbols of order zero, $A_1(\xi,t)$ and $A_2(\xi,t)$, defined as
  \begin{equation*}
    A_1(\xi,t):= \frac{2^{(1/2+\sigma)j}\beta(|2^{-j}t\xi|)}{(t|\xi|)^{1/2+\sigma}}\sum_{\ell=0}^Na_\ell\cdot(t|\xi|)^{-\ell},
  \end{equation*}
\begin{equation*}
    A_2(\xi,t):= \frac{2^{(1/2+\sigma)j}\beta(|2^{-j}t\xi|)}{(t|\xi|)^{1/2+\sigma}}\sum_{\ell=0}^Nb_\ell\cdot(t|\xi|)^{-\ell}.
  \end{equation*}
  Furthermore, we decompose the operator $\mathcal{A}_{t,j}^\sigma$ into three parts,  $\mathcal{A}_{t,j}^\sigma=\sum_{i=1}^3\mathcal{A}_{t,j,i}^\sigma$, where
  \begin{equation*}
    \mathcal{A}_{t,j,1}^\sigma f(x)=2^{-(1/2+\sigma)j}\int_{\mathbb{R}^2}e^{ix\xi+it|\xi|}A_1(\xi,t)\hat{f}(\xi)d\xi,
  \end{equation*}
  \begin{equation*}
    \mathcal{A}_{t,j,2}^\sigma f(x)=2^{-(1/2+\sigma)j}\int_{\mathbb{R}^2}e^{ix\xi-it|\xi|}A_2(\xi,t)\hat{f}(\xi)d\xi,
  \end{equation*}
  \begin{equation*}
    \mathcal{A}_{t,j,3}^\sigma f(x)=\int_{\mathbb{R}^2}e^{ix\xi}|t\xi|^{-\sigma}R_{\sigma,N}(t|\xi|)\beta(|2^{-j}t\xi|)\hat{f}(\xi)d\xi.
  \end{equation*}
Therefore, it suffices to show for $i=1,2,3$, and $\Re\sigma>-\kappa(p)$, one may find a positive real number $\varepsilon(p)$ such that
\begin{equation}\label{the1.3.1}
 \Big\|{\sup_{1<t<2}}^+\mathcal{A}_{t,j,i}^\sigma f\Big\|_p\lesssim 2^{-\varepsilon(p)j}\|f\|_{L_p}.
\end{equation}
For $i=3$, the inequality \eqref{the1.3.1} follows from the same argument as in Lemma \ref{leA_0}  once we apply the decay property \eqref{decay} of $R_{\sigma,N}$.

We only give the proof of \eqref{the1.3.1} for $i=1$ below since the proof for $i=2$ is similar.
Let $\rho_0,\rho_1\in C_0^\infty\big((1/8,8))$ and $\rho_0(r)=\rho_1(r)=1$ when $r\in(1/4,4)$. Then
\begin{equation*}
    \mathcal{A}_{t,j,1}^\sigma f(x)=2^{-(1/2+\sigma)j}\rho_1(t)\int_{\mathbb{R}^2}e^{ix\xi+it|\xi|}\rho_0(|2^{-j}\xi|)A_1(\xi,t)\hat{f}(\xi)d\xi.
  \end{equation*}
  Employing Proposition \ref{theorem3} with $\mu>\frac{1}{2}-\frac{1}{p}-\kappa(p)$, we establish the estimate
  \begin{equation}\label{Atji}
   \Big\|\mathcal{A}_{t,j,1}^\sigma f\Big\|_{L_p(L_\infty(\mathbb{R}^3)\overline{\otimes}\mathcal{M})}\leq C_\mu2^{(\mu-1/2-\Re\sigma)j}\|f\|_{L_p(L_\infty(\mathbb{R}^2)\overline{\otimes}\mathcal{M})}.
  \end{equation}
  Next, we define $\tilde{A_{1}}(\xi)(t)=2^{-j}\big(\frac{d}{dt}\big)(e^{it|\xi|}A_1(\xi,t))$, which is a symbol of order zero.
 Applying Proposition \ref{theorem3} once more, we derive
 \begin{equation}\label{Atji1}
   \Big\|\partial_t\big(\mathcal{A}_{t,j,1}^\sigma f\big)\Big\|_{L_p(L_\infty(\mathbb{R}^3)\overline{\otimes}\mathcal{M})}\leq C_\mu2^{(\mu+1/2-\Re\sigma)j}\|f\|_{L_p(L_\infty(\mathbb{R}^2)\overline{\otimes}\mathcal{M})}.
  \end{equation}

   To proceed with the arguments, we choose $I=(1,2)$ and $\lambda=2^{-j}$ in Lemma \ref{remax}. Combining the inequalities \eqref{Atji} and \eqref{Atji1}, and obtain
  \begin{equation*}
   \Big\|{\sup_{1<t<2}}^+\mathcal{A}_{t,j,1}^\sigma f\Big\|_p\lesssim C_\mu 2^{(\mu+1/p-1/2-\Re\sigma)j}\|f\|_p,
  \end{equation*}
 which holds for all $\mu>\frac{1}{2}-\frac{1}{p}-\kappa(p)$. When $\Re\sigma>-\kappa(p)$, we can choose suitable $\mu$ such that $\mu+1/p-1/2-\Re\sigma=\varepsilon(p)>0$. Consequently, we arrive at
\begin{equation*}
 \Big\|{\sup_{1<t<2}}^+\mathcal{A}_{t,j,1}^\sigma f\Big\|_p\lesssim 2^{-\varepsilon(p)j}\|f\|_{p}.
\end{equation*}
This completes the proof of Proposition \ref{theorem2} for the case $i=1$.

\subsection{Proof of Theorem \ref{thmloc}}\label{Se3.2}
In this subsection, we will show that Proposition \ref{theorem3} implies Theorem \ref{thmloc}. Employing Lemma \ref{Bessel}, we express the solution to the Cauchy problem of the wave equation as
\begin{equation*}
  u(x,t)=c_0\mathcal{A}_{t}^{-1/2}f(x)+c_1t\mathcal{A}_t^{1/2}g(x)
\end{equation*}
where $c_0$ and $c_1$ are universal constants. By applying the triangle inequality, we derive the inequality
\begin{equation*}
  \|u\|_{L_p(L_\infty([1,2]\times\mathbb{R}^2)\overline{\otimes}\mathcal{M})}
  \lesssim\Big\|\mathcal{A}_{t}^{-1/2}f\Big\|_{L_p(L_\infty([1,2]\times\mathbb{R}^2)\overline{\otimes}\mathcal{M})}+ \Big\|t\mathcal{A}_t^{1/2}g\Big\|_{L_p(L_\infty([1,2]\times\mathbb{R}^2)\overline{\otimes}\mathcal{M})}.
\end{equation*}
Next, it suffices to demonstrate that for $\nu>\frac{1}{2}-\frac{1}{p}-\kappa(p)$, we have
\begin{equation}\label{eq126}
  \Big\|\mathcal{A}_{t}^{-1/2}f\Big\|_{L_p(L_\infty([1,2]\times\mathbb{R}^2)\overline{\otimes}\mathcal{M})}
  \lesssim\|(I+\Delta)^{\nu/2}f\|_{L_p(L_\infty(\mathbb{R}^2)\overline{\otimes}\mathcal{M})}
\end{equation}
and
\begin{equation}\label{eq127}
  \Big\|t\mathcal{A}_t^{1/2}g\Big\|_{L_p(L_\infty([1,2]\times\mathbb{R}^2)\overline{\otimes}\mathcal{M})}
  \lesssim\|(I+\Delta)^{(\nu-1)/2}g\|_{L_p(L_\infty(\mathbb{R}^2)\overline{\otimes}\mathcal{M})}.
\end{equation}
Furthermore, adopting the methodology employed in the proof of Theorem \ref{theorem1}, it suffices to show that for $\nu>\frac{1}{2}-\frac{1}{p}-\kappa(p)$, we have
\begin{equation*}
  \sum_{j=0}^\infty\Big\|\mathcal{A}_{t,j}^{-1/2}f\Big\|_{L_p(L_\infty([1,2]\times\mathbb{R}^2)\overline{\otimes}\mathcal{M})}
  \lesssim\|(I+\Delta)^{\nu/2}f\|_{L_p(L_\infty(\mathbb{R}^2)\overline{\otimes}\mathcal{M})}
\end{equation*}
and
\begin{equation*}
  \sum_{j=0}^\infty\Big\|t\mathcal{A}_{t,j}^{1/2}g\Big\|_{L_p(L_\infty([1,2]\times\mathbb{R}^2)\overline{\otimes}\mathcal{M})}
  \lesssim\|(I+\Delta)^{(\nu-1)/2}g\|_{L_p(L_\infty(\mathbb{R}^2)\overline{\otimes}\mathcal{M})}.
\end{equation*}
Given that the proof in Lemma \ref{leA_0} still work here, we are reduced to showing that for a large $j>0$ with $\nu=\frac{1}{2}-\frac{1}{p}-\kappa(p)+2\varepsilon$, the inequalities
\begin{equation}\label{loc11}
  \Big\|\mathcal{A}_{t,j}^{-1/2}\big((I+\Delta)^{-\nu/2}f\big)\Big\|_{L_p(L_\infty([1,2]\times\mathbb{R}^2)\overline{\otimes}\mathcal{M})}\lesssim2^{-\varepsilon j}
  \|f\|_{L_p(L_\infty(\mathbb{R}^2)\overline{\otimes}\mathcal{M})}
\end{equation}
and
\begin{equation}\label{loc12}
  \Big\|t\mathcal{A}_{t,j}^{1/2}\big((I+\Delta)^{-(\nu-1)/2}g\big)\Big\|_{L_p(L_\infty([1,2]\times\mathbb{R}^2)\overline{\otimes}\mathcal{M})}\lesssim2^{-\varepsilon j}
 \|g\|_{L_p(L_\infty(\mathbb{R}^2)\overline{\otimes}\mathcal{M})}
\end{equation}
hold for all $\varepsilon>0$.
Here, we focus on the proof of inequality \eqref{loc12}, as the proof for inequality \eqref{loc11} can be done similarly.  Recall the expression for $ m_{\frac{1}{2}}(t|\xi|)$, given by
\begin{equation*}
  m_{\frac{1}{2}}(t|\xi|)=c\big(t|\xi|\big)^{-1}\sin(t|\xi|)=\frac{c}{2i}\big(t|\xi|\big)^{-1}\big(e^{it|\xi|}-e^{-it|\xi|}\big).
\end{equation*}
Thus we have
\begin{align}\label{local21}
  t\mathcal{A}_{t,j}^{1/2}\big((I+\Delta)^{-(\nu-1)/2}g\big)(x)& =c t\int_{\mathbb{R}^2}e^{ix\xi}(1+|\xi|^2)^{-(\nu-1)/2}m_{\frac{1}{2}}(t|\xi|)\beta(|2^{-j}t\xi|)\hat{g}(\xi)\ d\xi\nonumber\\
  &=\frac{c}{2i}\int_{\mathbb{R}^2}e^{ix\xi}\big(e^{it|\xi|}-e^{-it|\xi|}\big)\frac{\beta(|2^{-j}t\xi|)\hat{g}(\xi)}{|\xi|(1+|\xi|^2)^{(\nu-1)/2}}\ d\xi\nonumber\\
  &=\frac{c}{2i}\int_{\mathbb{R}^2}e^{ix\xi}\big(e^{it|\xi|}-e^{-it|\xi|}\big)2^{-\nu j}a(\xi,t)\hat{g}(\xi)\ d\xi,
\end{align}
where $$a(\xi,t)=\frac{2^{\nu j}\beta(|2^{-j}t\xi|)}{|\xi|(1+|\xi|^2)^{(\nu-1)/2}}$$ is a symbol of order zero. Next, let $\rho_0,\rho_1\in C_0^\infty\big((1/8,8))$ and $\rho_0(r)=\rho_1(r)=1$ when $r\in(1/4,4)$. For $t\in[1,2]$, we can rewrite the integral as
\begin{equation*}
  t\mathcal{A}_{t,j}^{1/2}\big((I+\Delta)^{-(\nu-1)/2}g\big)(x)=\frac{c2^{-\nu j}\rho_1(t)}{2i}\int_{\mathbb{R}^2}e^{ix\xi}\big(e^{it|\xi|}-e^{-it|\xi|}\big)\rho_0(|2^{-j}\xi|)a(\xi,t)\hat{g}(\xi)\ d\xi.
\end{equation*}

 Combining this with Proposition \ref{theorem3}, and choosing $\mu=\frac{1}{2}-\frac{1}{p}-\kappa(p)+\varepsilon$, we obtain
\begin{align*}
 \Big\|t\mathcal{A}_{t,j}^{1/2}\big((I+\Delta)^{-(\nu-1)/2}g\big)\Big\|_{L_p(L_\infty([1,2]\times\mathbb{R}^2)\overline{\otimes}\mathcal{M})}&\lesssim2^{-\nu j+\mu j}\|g\|_{L_p(L_\infty(\mathbb{R}^2)\overline{\otimes}\mathcal{M})}\\&\leq 2^{-\varepsilon j}
 \|g\|_{L_p(L_\infty(\mathbb{R}^2)\overline{\otimes}\mathcal{M})}.
\end{align*}
This completes the proof of Theorem \ref{thmloc}.

\subsection{Proof of Proposition \ref{theorem3}}\label{Se3.4}
In this subsection, we commence by proving Proposition \ref{theorem3} specifically for the cases $p=2$ and $p=\infty$.
Employing the Plancherel theorem twice, we have
\begin{align*}
  \|F_jf\|_{L_2\left(L_\infty(\mathbb{R}^3)\overline{\otimes}\mathcal{M}\right)}^2&=\tau\Big(\int_{\mathbb{R}^3}\Big|\rho_1(t)\rho_0(|2^{-j}\xi|)a(\xi,t)\hat{f}(\xi)\Big|^2d\xi dt\Big) \\
  &\lesssim\int_{1}^2\tau\Big(\int_{\mathbb{R}^2}|\hat{f}(\xi)|^2d\xi \Big)dt =\|f\|_{L_2\left(L_\infty(\mathbb{R}^2)\overline{\otimes}\mathcal{M}\right)}^2.
  \end{align*}
On the other hand, we express $ F_jf(x,t)$ as a convolution integral,
\begin{align*}
  F_jf(x,t)&=\rho_1(t)\int_{\mathbb{R}^2}e^{i(x\xi+t|\xi|)}\rho_0(|2^{-j}\xi|)a(\xi,t)\hat{f}(\xi) d\xi\\
  &=\int_{\mathbb{R}^2}\int_{\mathbb{R}^2}e^{i((x-y)\xi+t|\xi|)}\rho_1(t)\rho_0(|2^{-j}\xi|)a(\xi,t)f(y)d\xi dy
  \\&\triangleq \int_{\mathbb{R}^2}K_t(x-y)f(y)dy.
\end{align*}
Notice that the kernel satisfies $\|K_t\|_{L_1(\mathbb{R}^2)}\leq C2^{j/2}$ uniformly in $t$ (see \cite[Page 406-409]{St93} for further details). Consequently, by the Young inequality,
\begin{equation*}
 \|F_jf\|_{L_\infty\left(L_\infty(\mathbb{R}^3)\overline{\otimes}\mathcal{M}\right)}=\|K_t\ast f\|_{L_\infty\left(L_\infty(\mathbb{R}^3)\overline{\otimes}\mathcal{M}\right)}\leq\|f\|_\infty\sup_{t\in \mathbb{R}}\|K_t\|_1\leq C2^{j/2}\|f\|_\infty.
\end{equation*}
 To handle the general case for $p$, by interpolation, it suffices to establish a crucial $L^4$ estimate which is the key estimate among this paper.

\begin{proposition}\label{key}
Fix $\rho_0,\rho_1\in C_0^\infty\big((1/8,8)\big)$, let $a(\xi,t)$ be a symbol of order zero, i.e. for any $t\in(1,2), \alpha\in\mathbb{N}^2$,
 \begin{equation*}
   \Big|\left(\frac{\partial}{\partial \xi}\right)^\alpha a(\xi,t)\Big|\leq c_\alpha(1+|\xi|)^{-|\alpha|}.
 \end{equation*}
Define $$F_jf(x,t):=\rho_1(t)\int_{\mathbb{R}^2}e^{i(x\xi+t|\xi|)}\rho_0(|2^{-j}\xi|)a(\xi,t)\hat{f}(\xi) d\xi$$
 for $j\geq1$. Then we have
 \begin{equation*}
  \|F_jf\|_{L_4\left(L_\infty(\mathbb{R}^3)\overline{\otimes}\mathcal{M}\right)}\leq C_\mu2^{\mu j}\|f\|_{L_4\left(L_\infty(\mathbb{R}^2)\overline{\otimes}\mathcal{M}\right)}, \quad for\ all\ \mu>1/8.
   \end{equation*}
\end{proposition}
In the subsequent sections of this paper, our primary focus will be directed towards the rigorous proof of Proposition \ref{key}.

\section{Proof of Proposition \ref{key}}\label{Se4}
In this section, we fix a sufficiently large positive integer $j$ in Proposition \ref{key}, and omit the subscript $j$, equivalently denote $F(f)=F_j(f)$, for simplicity; similarly, starting with $F_j(f)$, we will need an angular decomposition $F_{j,v}(f)$ which is going to be simplified as $F_v(f)$ in \eqref{Fv}. Without loss of generality, we assume that the symbol $a(\xi,t)$ in Proposition \ref{key} is supported in the first quadrant with respect to the variable $\xi$. Our goal is to establish the following $L_4$-estimate:
\begin{equation}\label{L4}
  \|F(f)\|_{L_4\left(L_\infty(\mathbb{R}^3)\overline{\otimes}\mathcal{M}\right)}\leq C_\mu2^{\mu j}\|f\|_{L_4\left(L_\infty(\mathbb{R}^2)\overline{\otimes}\mathcal{M}\right)}, \qquad \text{for all}\ \mu>1/8,
\end{equation}
where $$F(f)(x,t)=\rho_1(t)\int_{\mathbb{R}^2}e^{i(x\xi+t|\xi|)}\rho_0(|2^{-j}\xi|)a(\xi,t)\hat{f}(\xi) d\xi$$ and $\rho_0,\rho_1\in C_0^\infty\big((1/8,8)\big)$.
We begin by rewriting $F(f)(x,t)$ as follows:
\begin{align*}
F(f)(x,t)&=\int_{\mathbb{R}}\int_{\mathbb{R}^2}e^{i(x\xi+t\tau)}\big(\rho_1(\cdot)a(\xi,\cdot)\big)^{\vee}(\tau-|\xi|)\rho_0(|2^{-j}\xi|)\hat{f}(\xi)\ d\xi\ d\tau\\
 &=\int_{\mathbb{R}^2}K(x-y,t)f(y)dy=\int_{\mathbb{R}^2}K_I(x-y,t)f(y)dy+\int_{\mathbb{R}^2}K_{II}(x-y,t)f(y)dy,
\end{align*}
where
\begin{align*}
 & K(x,t)=\int_{\mathbb{R}}\int_{\mathbb{R}^2}e^{i(x\xi+t\tau)}\big(\rho_1(\cdot)a(\xi,\cdot)\big)^{\vee}(\tau-|\xi|)\rho_0(|2^{-j}\xi|)\ d\xi\ d\tau,\\
 & K_I(x,t)=\int_{|\tau-2^j|\leq \frac{2^j}{10}}\int_{\mathbb{R}^2}e^{i(x\xi+t\tau)}\big(\rho_1(\cdot)a(\xi,\cdot)\big)^{\vee}(\tau-|\xi|)\rho_0(|2^{-j}\xi|)\ d\xi\ d\tau,\\
 & K_{II}(x,t)=\int_{|\tau-2^j|> \frac{2^j}{10}}\int_{\mathbb{R}^2}e^{i(x\xi+t\tau)}\big(\rho_1(\cdot)a(\xi,\cdot)\big)^{\vee}(\tau-|\xi|)\rho_0(|2^{-j}\xi|)\ d\xi\ d\tau.\\
\end{align*}
Using integration by parts, we obtain the following estimate:
\begin{equation}\label{K2}
  |K_{II}(x,t)|\lesssim_N\frac{2^{-jN}}{(1+|(x,t)|^2)^N}
\end{equation}
for any $N>0$ (see \cite[Page 81-88]{Sogge17} or \cite[Page 98]{Li18} for further details). Utilizing the Young inequality, we can further deduce that:
\begin{equation*}
  \Big\|\int_{\mathbb{R}^2}K_{II}(x-y,t)f(y)dy\Big\|_{_{L_4\left(L_\infty(\mathbb{R}^3)\overline{\otimes}\mathcal{M}\right)}}\lesssim_N 2^{-jN}\|f\|_{L_4\left(L_\infty(\mathbb{R}^2)\overline{\otimes}\mathcal{M}\right)}.
\end{equation*}
Given this result, in the remainder of this section, it suffices to take care of the estimates related to the kernel $K_I$, and one may thus assume that $\tau\thickapprox 2^{j}$, which will simplify our task of proving the estimate \eqref{L4}.
\subsection{Some lemmas}
In this subsection, we will introduce some frequency decompositions and prove some lemmas.

Let $\phi$ be a smoothing function on $\mathbb{R}$ with support in $[-1,1]$ such that $\sum_{n\in\mathbb{Z}}\phi^2(\cdot-n)=1$. Additionally, let $\rho\in C_0^\infty\big((1/8,8)\big)$ be a function that equals $1$ in the support of $\rho_0$. We define a decomposition in $\tau$, the variable dual to $t$, by introducing the operator $P^n$ on operator-valued functions in $\mathbb{R}^3$  as follows:
\begin{equation}\label{Pn}
  (P^ng)^{\wedge}(\xi,\tau):=\rho(2^{-j}|\xi|)\phi(2^{-j/2}\tau-n)\hat{g}(\xi,\tau).
\end{equation}
We denote the set $\{n\in\mathbb{Z}:n\thickapprox 2^{j/2}\}$ by $\Gamma_j$, then $|\Gamma_j|\thickapprox 2^{j/2}$.
Similarly, for $m=(m_1,m_2)\in\mathbb{Z}^2$, we define the operator $P_m$ acting on operator-valued functions in $\mathbb{R}^2$ as:
\begin{equation}\label{Pm}
  (P_mf)^{\wedge}(\xi):=\rho(2^{-j}|\xi|)\phi(2^{-j/2}\xi_1-m_1)\phi(2^{-j/2}\xi_2-m_2)\hat{f}(\xi),
\end{equation}
and
\begin{equation}\label{fn}
  f_n:=\sum_{\big\{m\in\mathbb{Z}^2:\big||m|-n\big|<100\big\}}P_mf.
\end{equation}

\begin{lemma}\label{lpn}
 Let $g=(g_{n,v})_{n\in\Gamma_j,v\in I}$ be an element in $L_p(\mathcal{N};\ell_2^c)$, where $I$ is a finite index set and $2\leq p\leq\infty$,we have
\begin{equation}\label{lpne}
  \Big\|\Big(\sum_{v\in I}\big|\sum_{n\in\Gamma_j} P^n g_{n,v}\big|^2\Big)^{1/2}\Big\|_{L_p(\mathcal{N})}\lesssim2^{(1/4-1/2p)j}\Big\|\Big(\sum_{v\in I}\sum_{n\in\Gamma_j}|g_{n,v}|^2\Big)^{1/2}\Big\|_{L_p(\mathcal{N})}.
\end{equation}
\end{lemma}
\begin{proof}
  By the Plancherel theorem and the support property of $\phi$, we can establish the estimate \eqref{lpne} for $p=2$. Due to the complex interpolation (see Lemma \ref{sq1}), it suffices to prove the inequality \eqref{lpne} for $p=\infty$.
  Let $K^n$ be the kernel associated with $P^n$, such that $P^ng(x,t)=(K^n\ast g)(x,t)$. From \eqref{Pn}, we derive,
  \begin{equation*}
    K^n(x,t)=\int_{\mathbb{R}^3}e^{i(x\xi+t\tau)}\rho(2^{-j}|\xi|)\phi(2^{-j/2}\tau-n) d\xi d\tau.
  \end{equation*}
  By applying integration by parts, we obtain the bound $|K^n(x,t)|\lesssim \frac{2^{\frac{5j}{2}}}{(1+2^j|x|+2^{j/2}|t|)^3}\triangleq \widetilde{K}(x,t)$, which is integrable. By Lemma \ref{convex} and the Young inequality,
  \begin{align*}
  \Big\|\Big(\sum_{v\in I}\big|\sum_{n\in\Gamma_j} P^n g_{n,v}\big|^2\Big)^{1/2}\Big\|_{L_\infty(\mathcal{N})}
    &\leq  \Big\|\big|\Gamma_j\big|\sum_{v\in I}\sum_{n\in\Gamma_j} |P^n g_{n,v}|^2\Big\|_{L_\infty(\mathcal{N})}^{1/2}\\
     &\lesssim 2^{j/4} \Big\|\sum_{v\in I}\sum_{n\in\Gamma_j} |K^n\ast g_{n,v}|^2\Big\|_{L_\infty(\mathcal{N})}^{1/2}\\
     &\lesssim 2^{j/4} \Big\|\widetilde{K}\ast\Big(\sum_{v\in I}\sum_{n\in\Gamma_j} |g_{n,v}|^2\Big)\Big\|_{L_\infty(\mathcal{N})}^{1/2}\\
      &\lesssim 2^{j/4} \Big\|\sum_{v\in I}\sum_{n\in\Gamma_j} |g_{n,v}|^2\Big\|_{L_\infty(\mathcal{N})}^{1/2}.
  \end{align*}
Hence we complete the proof.
\end{proof}

\begin{lemma}\label{lpm}
Given  $P_m$ as defined in \eqref{Pm}, for $2\leq p\leq \infty$, we have
\begin{equation}\label{lpme}
  \Big\|\Big(\sum_{m\in\mathbb{Z}^2}\big| P_m f\big|^2\Big)^{1/2}\Big\|_{L_p\left(L_\infty(\mathbb{R}^2)\overline{\otimes}\mathcal{M}\right)}\lesssim\|f\|_{L_p\left(L_\infty(\mathbb{R}^2)\overline{\otimes}\mathcal{M}\right)}.
\end{equation}
\end{lemma}

\begin{proof}
  When $p=2$, the result follows from the Plancherel theorem. Therefore, due to interpolation, it suffices to prove the inequality for
  $p=\infty$. Recall that the definition of $P_m$ in \eqref{Pm} involves the expression
  \begin{equation*}
    (P_mf)^{\wedge}(\xi)=\rho(2^{-j}|\xi|)\phi(2^{-j/2}\xi_1-m_1)\phi(2^{-j/2}\xi_2-m_2)\hat{f}(\xi).
  \end{equation*}
  For simplicity, we may assume $(P_mf)^{\wedge}(\xi):=\phi(2^{-j/2}\xi_1-m_1)\phi(2^{-j/2}\xi_2-m_2)\hat{f}(\xi)$. Indeed, suppose
  $(\widetilde{P_m}f)^{\wedge}(\xi):=\phi(2^{-j/2}\xi_1-m_1)\phi(2^{-j/2}\xi_2-m_2)\hat{f}(\xi)$, using the same technique as in the proof of Lemma \ref{lpn}, we can deduce that
  \begin{equation*}
    \Big\|\Big(\sum_{m\in\mathbb{Z}^2}\big| P_m f\big|^2\Big)^{1/2}\Big\|_{L_p\left(L_\infty(\mathbb{R}^2)\overline{\otimes}\mathcal{M}\right)}\lesssim\Big\|\Big(\sum_{m\in\mathbb{Z}^2}\big| \widetilde{P_m} f\big|^2\Big)^{1/2}\Big\|_{L_p\left(L_\infty(\mathbb{R}^2)\overline{\otimes}\mathcal{M}\right)}.
  \end{equation*}
Therefore,
\begin{align*}
    |P_mf(0)|&=\Big|\int_{\mathbb{R}^2}\phi(2^{-j/2}\xi_1-m_1)\phi(2^{-j/2}\xi_2-m_2)\hat{f}(\xi)d\xi\Big| \\&=\Big|\int_{\mathbb{R}^2}\int_{\mathbb{R}^2}e^{-ix\xi}\phi(2^{-j/2}\xi_1-m_1)\phi(2^{-j/2}\xi_2-m_2)f(x)d\xi dx\Big|\\&=\Big|\int_{\mathbb{R}^2}e^{-ix\cdot m}\hat{\phi}(x_1)\hat{\phi}(x_2)f(2^{-j/2}x)dx\Big|\\
    &=\Big|\sum_{n\in\mathbb{Z}^2}\int_{{[-\pi,\pi]}^2}e^{-ix\cdot m}h_n(x)dx\Big|=\Big|\sum_{n\in\mathbb{Z}^2}\hat{h}_n(m)\Big|,
\end{align*}
where $h_n(x)=\hat{\phi}(x_1-2n_1\pi)\hat{\phi}(x_2-2n_2\pi)f(2^{-j/2}(x-2\pi n))$.
Combining these results with the triangle inequality, we obtain
\begin{equation}\label{pme1}
\Big\|\Big(\sum_{m\in\mathbb{Z}^2}\big| P_m f(0)\big|^2\Big)^{1/2}\Big\|_{L_\infty(\mathcal{M})}
 \leq \sum_{n\in\mathbb{Z}^2}\Big\|\Big(\sum_{m\in\mathbb{Z}^2}\big|\hat{h}_n(m)\big|^2\Big)^{1/2}\Big\|_{L_\infty(\mathcal{M})}.
\end{equation}
Note that $L_\infty(\mathcal{M})\subset \mathcal{B}(L_2(\mathcal{M}))$, and applying the Parseval identity, we have
\begin{align*}
  \Big\|\Big(\sum_{m\in\mathbb{Z}^2}\big|\hat{h}_n(m)\big|^2\Big)^{1/2}\Big\|_{L_\infty(\mathcal{M})}^2&=\sup_{\|b\|_{L_2(\mathcal{M})}=1}\tau\big(\sum_{m\in\mathbb{Z}^2}\big|\hat{h}_n(m)b\big|^2\big)\\
  &=\sup_{\|b\|_{L_2(\mathcal{M})}=1}\tau\Big(\int_{{[-\pi,\pi]}^2}\big|h_n(x)b\big|^2dx\Big)\\
  &\leq \Big\|\int_{{[-\pi,\pi]}^2}\big|h_n(x)\big|^2dx\Big\|_{L_\infty(\mathcal{M})}.
\end{align*}
Using the inequality \eqref{pme1}, we can derive the following:
\begin{align*}
 &\Big\|\Big(\sum_{m\in\mathbb{Z}^2}\big| P_m f(0)\big|^2\Big)^{1/2}\Big\|_{L_\infty(\mathcal{M})}\leq
 \sum_{n\in\mathbb{Z}^2}\Big\|\int_{{[-\pi,\pi]}^2}\big|h_n(x)\big|^2dx\Big\|_{L_\infty(\mathcal{M})}^{1/2}\\
 &\leq
 \sum_{n\in\mathbb{Z}^2}\Big(\int_{{[-\pi,\pi]}^2}\big|\hat{\phi}(x_1-2n_1\pi)\hat{\phi}(x_2-2n_2\pi)\big|^2dx\Big)^{1/2}\times\|f\|_{L_\infty\left(L_\infty(\mathbb{R}^2)\overline{\otimes}\mathcal{M}\right)}
 \\&\lesssim \|f\|_{L_\infty\left(L_\infty(\mathbb{R}^2)\overline{\otimes}\mathcal{M}\right)}.
\end{align*}
Here, the final inequality holds due to the fact that $\hat{\phi}$ is a Schwartz function, which ensures that the sum over $n\in\mathbb{Z}^2$ is convergent.
Similarly, for any $y\in\mathbb{R}^2$, we have
\begin{equation*}
 \Big\|\Big(\sum_{m\in\mathbb{Z}^2}\big| P_m f(y)\big|^2\Big)^{1/2}\Big\|_{L_\infty(\mathcal{M})}\lesssim \|f\|_{L_\infty\left(L_\infty(\mathbb{R}^2)\overline{\otimes}\mathcal{M}\right)},
\end{equation*}
which implies the inequality \eqref{lpme} for $p=\infty$.
\end{proof}

We introduce the angular decomposition of the $\xi$-space, as outlined in \cite[Page 405]{St93}. Consider a roughly equally spaced set of points on the unit circle $S^1$ with grid length $2^{-j/2}$. Specifically, the set $\{\xi_v\}_{v=1}^{N(j)}$ comprises the unit vectors that adhere to the following conditions:
 \begin{itemize}
   \item [(1)] $|\xi_{v_1}-\xi_{v_2}|\geq 2^{-j/2}$, if $v_1\neq v_2$;
   \item [(2)] For any $\xi\in{\mathbb{R}^2\backslash 0}$, there exists one $\xi_v$, such that $\big|{\xi}/{|\xi|}-\xi_v\big|<2^{-j/2}$.
 \end{itemize}
 It is evident that $N(j)\thickapprox2^{j/2}$, and we can assume that $\arg\xi_v=2\pi v2^{-j/2}$.
Let $\Omega_v$ denote the cone in the $\xi$-space with the central direction $\xi_v$, defined as:
$$\Omega_v=\{\xi\in{\mathbb{R}^2\backslash \{0\}}: \big|{\xi}/{|\xi|}-\xi_v\big|<2\cdot2^{-j/2}\}.$$
Next, we construct an associated partition of unity: we select a family of smoothing functions $\chi_v$ that satisfy
$\sum_{v=1}^{N(j)}\chi_v(\xi)=1$ for all $\xi\in{\mathbb{R}^2\backslash \{0\}}$  and possess the following properties:
\begin{itemize}
  \item [(1)] $\chi_v$ is homogeneous of degree $0$, meaning $\chi_v(\xi)=\chi_{v}(\xi/|\xi|)$ for all $\xi\in{\mathbb{R}^2\backslash \{0\}}$;
  \item [(2)] $\chi_v$ is supported in $\Omega_v$ and fulfills the uniform estimate:
  \begin{equation*}
    |\partial_{\xi}^\alpha\xi_v(\xi)|\leq C_\alpha2^{|\alpha|j/2|}|\xi|^{-|\alpha|}.
  \end{equation*}
\end{itemize}
We now define the operators as follows:
\begin{equation}\label{Fv}
  F_v(f)(x,t):=\rho_1(t)\int_{\mathbb{R}^2}e^{i(x\xi+t|\xi|)}\rho_0(|2^{-j}\xi|)a(\xi,t)\chi_v(\xi)\hat{f}(\xi) d\xi
\end{equation}
and
\begin{equation}\label{Qv}
  (Q_vg)^{\wedge}(\xi,\tau):=\psi\left(\frac{|\xi|-\tau}{|\xi|^\epsilon}\right)\chi_v(\xi)\hat{g}(\xi,\tau)\triangleq\Psi_v(\xi,\tau)\hat{g}(\xi,\tau),
\end{equation}
where $\psi$ is a smoothing function on $\mathbb{R}$ supported in $[-2,2]$ and equal to $1$ in $[-1,1]$, and $\epsilon$ is a suitably chosen small positive number.

Prior to proving Proposition \ref{key}, we present two lemmas to address the error terms.
\begin{lemma}\label{er1}
  For $1\leq p\leq \infty$ and every $N\in\mathbb{N}$,
  \begin{equation}\label{er1e}
    \Big\|\sum_{n\in\mathbb{Z}}P^nF(f-f_n)\Big\|_{L_p\left(L_\infty(\mathbb{R}^3)\overline{\otimes}\mathcal{M}\right)}\lesssim_N2^{-Nj}\|f\|_{L_p\left(L_\infty(\mathbb{R}^2)\overline{\otimes}\mathcal{M}\right)}.
  \end{equation}
\end{lemma}
\begin{proof}Fix an integer $n$. Then,
\begin{align*}
P^nF(f-f_n)(x,t)&=\int_{\mathbb{R}}\int_{\mathbb{R}^2}e^{i(x\xi+t\tau)}\big(\rho_1(\cdot)a(\xi,\cdot)\big)^{\vee}(\tau-|\xi|)\rho_0(|2^{-j}\xi|)
\\&\qquad \times\phi(2^{-j/2}\tau-n)(1-\tilde{\phi}(\xi))\hat{f}(\xi)\ d\xi\ d\tau\\
&\triangleq \int_{\mathbb{R}^2}\widetilde{K}_n(x-y,t)f(y)dy,
\end{align*}
where the kernel $\widetilde{K}_n(x,t)$ is given by
\begin{equation*}
  \int_{\mathbb{R}}\int_{\mathbb{R}^2}e^{i(x\xi+t\tau)}\big(\rho_1(\cdot)a(\xi,\cdot)\big)^{\vee}(\tau-|\xi|)\rho_0(|2^{-j}\xi|)\phi(2^{-j/2}\tau-n)(1-\tilde{\phi}(\xi))\ d\xi\ d\tau
\end{equation*}
and\begin{equation*}
  \tilde{\phi}(\xi)=\sum_{\big\{m\in\mathbb{Z}^2:\big||m|-n\big|<100\big\}}\rho(2^{-j}|\xi|)\phi(2^{-j/2}\xi_1-m_1)\phi(2^{-j/2}\xi_2-m_2).
\end{equation*}
 Recall that we have assumed that $\tau\thickapprox 2^j$, and by the support of $\phi$, $\widetilde{K}_n=0$ if $n\notin\Gamma_j$. As the estimate \eqref{K2}, one has
 \begin{equation*}
  |\widetilde{K}_n(x,t)|\lesssim_N\frac{2^{-Nj}}{(1+|(x,t)|^2)^N}
\end{equation*}
for any $N>0$ (see \cite[Page 82]{Sogge17} for more details). Then, by the Young inequality,
\begin{equation*}
  \Big\|P^nF(f-f_n)\Big\|_{{L_p\left(L_\infty(\mathbb{R}^3)\overline{\otimes}\mathcal{M}\right)}}\lesssim_N 2^{-Nj}\|f\|_{L_p\left(L_\infty(\mathbb{R}^2)\overline{\otimes}\mathcal{M}\right)}.
\end{equation*}
 Therefore, noting that  $|\Gamma_j|\thickapprox 2^{j/2}$,
 \begin{align*}
   \Big\|\sum_{n\in\mathbb{Z}}P^nF(f-f_n)\Big\|_{L_p\left(L_\infty(\mathbb{R}^3)\overline{\otimes}\mathcal{M}\right)}&\leq \sum_{n\in\Gamma_j}\Big\|P^nF(f-f_n)\Big\|_{L_p\left(L_\infty(\mathbb{R}^3)\overline{\otimes}\mathcal{M}\right)}\\
   & \lesssim_N2^{-Nj}\|f\|_{L_p\left(L_\infty(\mathbb{R}^2)\overline{\otimes}\mathcal{M}\right)}.
 \end{align*}
\end{proof}

\begin{lemma}\label{er2}
Define  $R_v(f):=F_v(f)-Q_vF_v(f)$, then for any $N\in\mathbb{N}$ and $1\leq p\leq\infty$, we have
\begin{equation}\label{er2e}
  \|R_v(f)\|_{{L_p\left(L_\infty(\mathbb{R}^3)\overline{\otimes}\mathcal{M}\right)}}\lesssim_N 2^{-Nj}\|f\|_{L_p\left(L_\infty(\mathbb{R}^2)\overline{\otimes}\mathcal{M}\right)}.
\end{equation}
\end{lemma}

\begin{proof}
Note that
the Fourier transform of $F_v(f)$  is given by
\begin{equation*}
  \hat{f}(\xi)\rho_0(|2^{-j}\xi|)\chi_v(\xi)(\rho_1(\cdot)a(\xi,\cdot))^{\vee}(\tau-|\xi|).
\end{equation*}
Then, the kernel associated with $R_v$ is
\begin{equation*}
  K_v(x,t)=\int_{\mathbb{R}^3}e^{i(x\cdot\xi+t\tau)} \Big(1-\psi\big(\frac{|\xi|-\tau}{|\xi|^\epsilon}\big)\Big)\rho_0(|2^{-j}\xi|)\chi_v(\xi)(\rho_1(\cdot)a(\xi,\cdot))^{\vee}(\tau-|\xi|)d\xi d\tau.
\end{equation*}
As the estimate \eqref{K2}, one has
 \begin{equation*}
  |K_v(x,t)|\lesssim_N\frac{2^{-jN}}{(1+|(x,t)|^2)^N}
\end{equation*}
for any $N>0$ (see \cite[Page 84]{Sogge17} for further details). Now, we complete the proof as in Lemma \ref{er1}.
\end{proof}

\subsection{An $L_4$ estimate}
To prove the $L_4$-estimate \eqref{L4}, we start by noting that $$\sum_{n\in\Gamma_j}(P^n)^2F(f)=F(f).$$ Applying the triangle inequality and Lemma \ref{er1}, which also holds when $(P^n)^2$ is used instead of $P^n$, we have
\begin{align*}
  \|F(f)\|_{L_4(\mathcal{N})}&=\Big\|\sum_{n\in\Gamma_j}(P^n)^2F(f)\Big\|_{L_4(\mathcal{N})}\\&\leq \Big\|\sum_{n\in\Gamma_j}(P^n)^2F(f_n)\Big\|_{L_4(\mathcal{N})}+\Big\|\sum_{n\in\Gamma_j}(P^n)^2F(f-f_n)\Big\|_{L_4(\mathcal{N})}\\
  &\lesssim \Big\|\sum_{n\in\Gamma_j}(P^n)^2F(f_n)\Big\|_{L_4(\mathcal{N})}
  +\|f\|_{L_4\left(L_\infty(\mathbb{R}^2)\overline{\otimes}\mathcal{M}\right)}.
\end{align*}
Next, let $I:=\{1\leq v\leq N(j): \xi_v \ \text{in the first quadrant}\}$ and recall that $a(\xi,t)$ in \eqref{Fv} is supported in the first quadrant with respect to the variable $\xi$. By the triangle inequality, the Young inequality and Lemma \ref{er2}, we obtain
\begin{align*}
  &\Big\|\sum_{n\in\Gamma_j}(P^n)^2F(f_n)\Big\|_{L_4(\mathcal{N})}=\Big\|\sum_{n\in\Gamma_j}\sum_{v\in I}(P^n)^2F_v(f_n)\Big\|_{L_4(\mathcal{N})}\\
  &\leq \Big\|\sum_{n\in\Gamma_j}\sum_{v\in I}(P^n)^2Q_vF_v(f_n)\Big\|_{L_4(\mathcal{N})}+
   \sum_{n\in\Gamma_j}\sum_{v\in I}\Big\|(P^n)^2R_v(f_n)\Big\|_{L_4(\mathcal{N})}\\
   &\leq \Big\|\sum_{n\in\Gamma_j}\sum_{v\in I}(P^n)^2Q_vF_v(f_n)\Big\|_{L_4(\mathcal{N})}+
   \sum_{n\in\Gamma_j}\sum_{v\in I}\Big\|R_v(f_n)\Big\|_{L_4(\mathcal{N})}\\
  &\lesssim \Big\|\sum_{n\in\Gamma_j}\sum_{v\in I}(P^n)^2Q_vF_v(f_n)\Big\|_{L_4(\mathcal{N})}
  +\|f\|_{L_4\left(L_\infty(\mathbb{R}^2)\overline{\otimes}\mathcal{M}\right)}.
\end{align*}
The final inequality above uses the facts that $|I|\thickapprox|\Gamma_j|\thickapprox2^{j/2}$ and
\begin{equation*}
  \|f_n\|_{L_4\left(L_\infty(\mathbb{R}^2)\overline{\otimes}\mathcal{M}\right)}=\Big\|\sum_{m\in\mathbb{Z}^2:\big||m|-n\big|<100}P_mf\Big\|_{L_4\left(L_\infty(\mathbb{R}^2)\overline{\otimes}\mathcal{M}\right)}
  \lesssim\|f\|_{L_4\left(L_\infty(\mathbb{R}^2)\overline{\otimes}\mathcal{M}\right)}.
\end{equation*}
It now suffices to show:
\begin{equation}\label{key1115}
  \Big\|\sum_{n\in\Gamma_j}\sum_{v\in I}(P^n)^2Q_vF_v(f_n)\Big\|_{L_4(\mathcal{N})}\lesssim_\mu2^{\mu j}\|f\|_{L_4\left(L_\infty(\mathbb{R}^2)\overline{\otimes}\mathcal{M}\right)},\qquad \text{for all}\ \mu>1/8.
\end{equation}
To proceed with the proof, we deduce that
\begin{align}\label{svw}
  &\Big\|\sum_{n\in\Gamma_j}\sum_{v\in I}(P^n)^2Q_vF_v(f_n)\Big\|_{L_4(\mathcal{N})}^4\nonumber\\ &=\varphi\Big(\Big|\sum_{v,w\in I}\big(\sum_{n\in\Gamma_j}((P^n)^2Q_vF_v(f_n))^*\big)\big(\sum_{n\in\Gamma_j}(P^n)^2Q_wF_w(f_n)\big)\Big|^2\Big)\nonumber\\
  &\lesssim (I)+(II),
\end{align}
where
\begin{equation*}
  (I)=\varphi\Big(\Big|\sum_{v,w\in I,|v-w|\leq1000}\big(\sum_{n\in\Gamma_j}((P^n)^2Q_vF_v(f_n))^*\big)\big(\sum_{n\in\Gamma_j}(P^n)^2Q_wF_w(f_n)\big)\Big|^2\Big)
\end{equation*}
and
\begin{equation*}
  (II)=\varphi\Big(\Big|\sum_{v,w\in I,|v-w|>1000}\big(\sum_{n\in\Gamma_j}((P^n)^2Q_vF_v(f_n))^*\big)\big(\sum_{n\in\Gamma_j}(P^n)^2Q_wF_w(f_n)\big)\Big|^2\Big).
\end{equation*}

 We first estimate part $(I)$. By the triangle inequality, and Lemma \ref{sq},
 \begin{align*}
   (I) &=\varphi\Big(\Big|\sum_{i=-1000}^{1000}\sum_{v,v+i\in I}\big(\sum_{n\in\Gamma_j}((P^n)^2Q_vF_v(f_n))^*\big)\big(\sum_{n\in\Gamma_j}(P^n)^2Q_{v+i}F_{v+i}(f_n)\big)\Big|^2\Big) \\
   &\lesssim \sum_{i=-1000}^{1000}\varphi\Big(\Big|\sum_{v,v+i\in I}\big(\sum_{n\in\Gamma_j}((P^n)^2Q_vF_v(f_n))^*\big)\big(\sum_{n\in\Gamma_j}(P^n)^2Q_{v+i}F_{v+i}(f_n)\big)\Big|^2\Big)\\
   &\lesssim \sum_{i=-1000}^{1000}\Big\|\Big(\sum_{v:v,v+i\in I}\big|\sum_{n\in\Gamma_j}(P^n)^2Q_vF_v(f_n)\big|^2\Big)^{1/2}\Big\|_{L_4(\mathcal{N})}^2\\&
   \qquad\times\Big\|\Big(\sum_{v+i:v,v+i\in I}\big|\sum_{n\in\Gamma_j}(P^n)^2Q_{v+i}F_{v+i}(f_n)\big|^2\Big)^{1/2}\Big\|_{L_4(\mathcal{N})}^2\\
   &\lesssim\Big\|\Big(\sum_{v\in I}\big|\sum_{n\in\Gamma_j}(P^n)^2Q_vF_v(f_n)\big|^2\Big)^{1/2}\Big\|_{L_4(\mathcal{N})}^4.
 \end{align*}
By applying Lemma \ref{lpn}, we arrive at
 \begin{equation}\label{esI}
   (I)\lesssim 2^{j/2}\Big\|\Big(\sum_{v\in I}\sum_{n\in\Gamma_j}\big|P^nQ_vF_v(f_n)\big|^2\Big)^{1/2}\Big\|_{L_4(\mathcal{N})}^4.
 \end{equation}

To estimate part $(II)$,  we require a geometric estimate. For each $n\in\Gamma_j,v\in I$, observe that the Fourier transform of $P^nQ_vF_v(f_n)$ is supported in the region
\begin{equation*}
  \mathcal{U}_{n}^v:=\{(\xi,\tau)\in\mathbb{R}^3:(\xi,\tau)\in\supp\Psi_v,|2^{-j/2}\tau-n|\leq1\},
\end{equation*}
where $\Psi_v$ was defined in \eqref{Qv}.
Next, we derive an estimate for the number of overlaps of algebraic differences between the sets $\mathcal{U}_{n}^{v}$ and $\mathcal{U}_{m}^{w}$:
\begin{lemma}\label{geo}
  There exists a constant $C$, independent of $j$, such that
  \begin{equation}\label{geoe}
    \sum_{n,m\in\Gamma_j}\sum_{v,w\in I,|v-w|>1000}\chi_{\mathcal{U}_{n}^v-\mathcal{U}_{m}^{w}}(\xi,\tau)\leq C2^{j/2+2\epsilon j}.
  \end{equation}
\end{lemma}

\begin{remark}\label{remark4}
  \rm{ Caused by the fact that $|x^*x|^2\neq|xx|^2$ for a general operator $x$, a geometric estimate for the form like $\sup_{(\xi,\tau)\in\mathbb{R}^3}\sum_{v,w\in I}\chi_{\mathcal{U}_{n}^v-\mathcal{U}_{m}^{w}}(\xi,\tau)$, which differs from the geometric estimate presented in the classical setting \cite{MSS}, has to be considered. Furthermore, we cannot anticipate the same upper bound as in \cite{MSS} for our new estimate, as $\sum_{v,w\in I}\chi_{\mathcal{U}_{n}^v-\mathcal{U}_{n^\prime}^{w}}(\xi,\tau)$ may become significant large near the origin.
   To overcome this obstacle, we observe that if $|v-w|\geq C$, where $C$ is a large positive constant, then $\sup_{(\xi,\tau)\in\mathbb{R}^3}\sum_{v,w\in I,|v-w|>C}\chi_{\mathcal{U}_{n}^v-\mathcal{U}_{n^\prime}^{w}}(\xi,\tau)$  admits a suitable upper bound. Consequently, we must partition the double sums over $v,w$ in the inequality \eqref{svw}. Therefore, in comparison to Lemma 1.2 in \cite{MSS}, we require a stronger version as Lemma \ref{lpn} to tackle the first part $(I)$.}
\end{remark}

The proof of Lemma \ref{geo} will be given in the subsequent section. For the term $(II)$, invoking the Plancherel theorem, we derive
\begin{align}\label{IIe1}
 (II)&=\varphi\Big(\Big|\sum_{v,w\in I,|v-w|>1000}\big(\sum_{n\in\Gamma_j}((P^n)^2Q_vF_v(f_n))^*\big)\big(\sum_{n\in\Gamma_j}(P^n)^2Q_wF_w(f_n)\big)\Big|^2\Big)\\
 &=\varphi\Big(\Big|\sum_{v,w\in I,|v-w|>1000}\sum_{n,n^\prime\in\Gamma_j}\big(((P^n)^2Q_vF_v(f_n))^*\big)^\wedge\ast\big((P^{n^\prime})^2Q_wF_w(f_{n^\prime})\big)^\wedge\Big|^2\Big)\nonumber\\
 &=\varphi\Big(\Big|\sum_{v,w\in I,|v-w|>1000}\sum_{n,n^\prime\in\Gamma_j}\chi_{\mathcal{U}_{n}^v-\mathcal{U}_{n^\prime}^{w}}\big(((P^n)^2Q_vF_v(f_n))^*\big)^\wedge\ast\big((P^{n^\prime})^2Q_wF_w(f_{n^\prime})\big)^\wedge\Big|^2\Big)\nonumber.
\end{align}
By Lemma \ref{convex},
\begin{align}\label{IIe2}
  &\Big|\sum_{v,w\in I,|v-w|>1000}\sum_{n,n^\prime\in\Gamma_j}\chi_{\mathcal{U}_{n}^v-\mathcal{U}_{n^\prime}^{w}}\big(((P^n)^2Q_vF_v(f_n))^*\big)^\wedge\ast\big((P^{n^\prime})^2Q_wF_w(f_{n^\prime})\big)^\wedge\Big|^2\nonumber\\
  &\leq\Big(\sum_{v,w\in I,|v-w|>1000}\sum_{n,n^\prime\in\Gamma_j}\Big|\big(((P^n)^2Q_vF_v(f_n))^*\big)^\wedge\ast\big((P^{n^\prime})^2Q_wF_w(f_{n^\prime})\big)^\wedge\Big|^2\Big)\nonumber
  \\&\quad\times\Big(\sum_{v,w\in I,|v-w|>1000}\sum_{n,n^\prime\in\Gamma_j}\chi_{\mathcal{U}_{n}^v-\mathcal{U}_{n^\prime}^{w}}\Big).
\end{align}
Combining the inequalities \eqref{IIe1}, \eqref{IIe2}, \eqref{geoe}, and applying once more the Plancherel theorem, we obtain:
\begin{align}\label{IIe3}
  (II)
  &\leq C2^{j/2+2\epsilon j} \varphi\Big(\sum_{v,w\in I,|v-w|>1000}\sum_{n,n^\prime\in\Gamma_j}\Big|((P^n)^2Q_vF_v(f_n))^*(P^{n^\prime})^2Q_wF_w(f_{n^\prime})\Big|^2\Big)\nonumber\\
  &\leq C2^{j/2+2\epsilon j} \varphi\Big(\sum_{v,w\in I}\sum_{n,n^\prime\in\Gamma_j}\Big|((P^n)^2Q_vF_v(f_n))^*(P^{n^\prime})^2Q_wF_w(f_{n^\prime})\Big|^2\Big)\nonumber\\
  &\leq C2^{j/2+2\epsilon j} \Big\|\Big(\sum_{v\in I}\sum_{n\in\Gamma_j}\big|\big((P^n)^2Q_vF_v(f_n)\big)^*\big|^2\Big)^{1/2}\Big\|_{L_4(\mathcal{N})}^4.
\end{align}

By leveraging the inequalities \eqref{key1115}, \eqref{svw}, \eqref{esI}, and \eqref{IIe3}, Proposition \ref{key} can be reduced to demonstrating the following inequalities:
\begin{equation}\label{last1}
  \Big\|\Big(\sum_{v\in I}\sum_{n\in\Gamma_j}\big|P^nQ_vF_v(f_n)\big|^2\Big)^{1/2}\Big\|_{L_4(\mathcal{N})}\leq C_\varepsilon 2^{\varepsilon j}\|f\|_{L_4\left(L_\infty(\mathbb{R}^2)\overline{\otimes}\mathcal{M}\right)},
\end{equation}
and
\begin{equation}\label{last2}
 \Big\|\Big(\sum_{v\in I}\sum_{n\in\Gamma_j}\big|\big((P^n)^2Q_vF_v(f_n)\big)^*\big|^2\Big)^{1/2}\Big\|_{L_4(\mathcal{N})}\leq C_\varepsilon 2^{\varepsilon j}\|f\|_{L_4\left(L_\infty(\mathbb{R}^2)\overline{\otimes}\mathcal{M}\right)}
\end{equation}
for $\varepsilon>0$.
Here, we will concentrate on proving the inequality \eqref{last1}, since the proof of the inequality \eqref{last2} is analogous.

Using Lemma \ref{er2}, it is sufficient to estimate the norm $$\Big\|\Big(\sum_{v\in I}\sum_{n\in\Gamma_j}\big|P^nF_v(f_n)\big|^2\Big)^{1/2}\Big\|_{L_4(\mathcal{N})}.$$
Indeed, by applying the triangle inequality alongside Lemma \ref{er2}, we derive the following inequality:
\begin{align*}
  & \Big\|\Big(\sum_{v\in I}\sum_{n\in\Gamma_j}\big|P^nQ_vF_v(f_n)\big|^2\Big)^{1/2}\Big\|_{L_4(\mathcal{N})}\\
  &\leq \Big\|\Big(\sum_{v\in I}\sum_{n\in\Gamma_j}\big|P^nF_v(f_n)\big|^2\Big)^{1/2}\Big\|_{L_4(\mathcal{N})}+\Big\|\Big(\sum_{v\in I}\sum_{n\in\Gamma_j}\big|P^nR_v(f_n)\big|^2\Big)^{1/2}\Big\|_{L_4(\mathcal{N})}\\
  &\lesssim \Big\|\Big(\sum_{v\in I}\sum_{n\in\Gamma_j}\big|P^nF_v(f_n)\big|^2\Big)^{1/2}\Big\|_{L_4(\mathcal{N})}+\|f\|_{L_4\left(L_\infty(\mathbb{R}^2)\overline{\otimes}\mathcal{M}\right)}.
\end{align*}
We assert that the following inequality holds:
\begin{equation}\label{last3}
  \Big\|\Big(\sum_{v\in I}\sum_{n\in\Gamma_j}\big|P^nF_v(f_n)\big|^2\Big)^{1/2}\Big\|_{L_4(\mathcal{N})}\lesssim\Big\|\Big(\sum_{v\in I}\sum_{n\in\Gamma_j}\big|F_v(f_n)\big|^2\Big)^{1/2}\Big\|_{L_4(\mathcal{N})}.
\end{equation}
To justify this claim, recall from the proof of Lemma \ref{lpn} that the kernel $K^n$ associated with the operator $P^n$ satisfies the bound
\begin{equation*}
  |K^n(x,t)|\leq C_N \frac{2^{\frac{5j}{2}}}{(1+2^j|x|+2^{j/2}|t|)^N}.
\end{equation*}
Combining this with Lemma \ref{convex}, and denoting $F_v(f_n)$ by $g_{v,n}$,
\begin{align}\label{last4}
  \big|P^nF_v(f_n)\big|^2&=\big|\int_{\mathbb{R}}\int_{\mathbb{R}^2}K^n(y,s)g_{v,n}(y-x,s-t)dyds\big|^2 \nonumber\\
&\leq \Big(\int_{\mathbb{R}}\int_{\mathbb{R}^2}|K^n(y,s)|dyds\Big)\Big(\int_{\mathbb{R}}\int_{\mathbb{R}^2}|K^n(y,s)||g_{v,n}(y-x,s-t)|^2dyds\Big)\nonumber\\
&\leq C_N\int_{\mathbb{R}}\int_{\mathbb{R}^2}\frac{2^{\frac{5j}{2}}|g_{v,n}(y-x,s-t)|^2}{(1+2^j|y|+2^{j/2}|s|)^N}dyds.
\end{align}
Utilizing the inequality \eqref{last4}, the triangle inequality, and the Cauchy-Schwarz inequality, we derive
 \begin{align*}
   &\Big\|\Big(\sum_{v\in I}\sum_{n\in\Gamma_j}\big|P^nQ_vF_v(f_n)\big|^2\Big)^{1/2}\Big\|_{L_4(\mathcal{N})}^4\\
   &\leq C_N\Big\|\sum_{v\in I}\sum_{n\in\Gamma_j}\int_{\mathbb{R}}\int_{\mathbb{R}^2}\frac{2^{\frac{5j}{2}}|g_{v,n}(y-x,s-t)|^2}{(1+2^j|y|+2^{j/2}|s|)^N}dyds\Big\|_{L_2(\mathcal{N})}^2\\
   &\leq C_N\Big(\int_{\mathbb{R}}\int_{\mathbb{R}^2}\frac{2^{\frac{5j}{2}}}{(1+2^j|y|+2^{j/2}|s|)^N}\Big\|\sum_{v\in I}\sum_{n\in\Gamma_j}|g_{v,n}(y-x,s-t)|^2\Big\|_{L_2(\mathcal{N})}dyds\Big)^2\\
  &\lesssim \Big\|\sum_{v\in I}\sum_{n\in\Gamma_j}|g_{v,n}|^2\Big\|_{L_2(\mathcal{N})}^2=\Big\|\Big(\sum_{v\in I}\sum_{n\in\Gamma_j}\big|F_v(f_n)\big|^2\Big)^{1/2}\Big\|_{L_4(\mathcal{N})}^4.
 \end{align*}
 This confirms our claim \eqref{last3}. Now, we proceed to give an estimate for the term $$\Big\|\Big(\sum_{v\in I}\sum_{n\in\Gamma_j}\big|F_v(f_n)\big|^2\Big)^{1/2}\Big\|_{L_4(\mathcal{N})}.$$ As in \cite{MSS},
for $v\in I, n\in\Gamma_j$, we define the set
 \begin{equation*}
   \mathcal{J}_{v,n}:=\{m\in\mathbb{Z}^2:\chi_v(P_mf)^\vee\neq0,\big||m|-n\big|\leq 100\}.
 \end{equation*}
 It follows that $|\mathcal{J}_{v,n}|\lesssim 1$ and, for $m\in\mathbb{Z}^2$, there exists an absolute constant $C$ such that
  \begin{equation}\label{Jvn1}
   \big| \{(v,n)\in I\times \Gamma_j:m\in\mathcal{J}_{v,n}\}\big|\leq C.
\end{equation}
Furthermore, we have $F_v(f_n)=F_v(f_{v,n})$, where $f_{v,n}=\sum_{m\in\mathcal{J}_{v,n}}P_mf$. Before presenting the final proof of Proposition \ref{key}, we introduce the following lemma, which will be proved in the subsequent section.
 \begin{lemma}\label{ka}
   For $v\in I,n\in\Gamma_j$ and $g\in L_2(\mathcal{N})$, let $K_{v,n}(y,t)$ be the kernel associated with $F_v(f_n)$. Then,
   \begin{equation}\label{ka1}
    \sup_{y,t,v,n}\int_{\mathbb{R}^2}|K_{v,n}(y-x,t)|dx\leq C
   \end{equation} and
   \begin{equation}\label{ka2}
     \Big\|{\sup_{v\in I, n\in\Gamma_j}}^+\int_{\mathbb{R}^3}|K_{v,n}(y-x,t)|g(y,t)dydt\Big\|_2\leq C_\varepsilon2^{\varepsilon j}\|g\|_{L_2(\mathcal{N})}
   \end{equation}
   for $\varepsilon>0$.
\end{lemma}

Applying Lemma \ref{convex} and the inequality \eqref{ka1}, we derive
\begin{align*}
 |F_v(f_n)(y,t)|^2&=\Big|\int_{\mathbb{R}^2}K_{v,n}(y-x,t)f_{v,n}(x)dx\Big|^2\\
 &\leq C \int_{\mathbb{R}^2}|K_{v,n}(y-x,t)||f_{v,n}(x)|^2dx.
\end{align*}
Combining this and using duality, we obtain
 \begin{align}\label{last5}
   &\Big\|\Big(\sum_{v\in I}\sum_{n\in\Gamma_j}\big|F_v(f_n)\big|^2\Big)^{1/2}\Big\|_{L_4(\mathcal{N})}^2=\Big\|\sum_{v\in I}\sum_{n\in\Gamma_j}\big|F_v(f_n)\big|^2\Big\|_{L_2(\mathcal{N})}\nonumber\\&\lesssim\sup_{\|g\|_2=1}\varphi\Big( \sum_{v\in I}\sum_{n\in\Gamma_j}\int_{\mathbb{R}^2}|K_{v,n}(y-x,t)||f_{v,n}(x)|^2dxg(y,t)\Big)\nonumber\\
   &=\sup_{\|g\|_2=1}\int_{\mathbb{R}^2}\tau\Big( \sum_{v\in I}\sum_{n\in\Gamma_j}|f_{v,n}(x)|^2\int_{\mathbb{R}^3}|K_{v,n}(y-x,t)|g(y,t)dydt\Big)dx\nonumber\\
   &\leq \sup_{\|g\|_2=1}\Big\| \sum_{v\in I}\sum_{n\in\Gamma_j}|f_{v,n}|^2\Big\|_2 \Big\|{\sup_{v\in I, n\in\Gamma_j}}^+\int_{\mathbb{R}^3}|K_{v,n}(y-x,t)|g(y,t)dydt\Big\|_2.
 \end{align}
 Note that by Lemma \ref{convex} and the fact $|\mathcal{J}_{v,n}|\lesssim 1$,
 \begin{equation}\label{Jvn2}
  |f_{v,n}|^2=\Big|\sum_{m\in\mathcal{J}_{v,n}}P_mf\Big|^2\lesssim \sum_{m\in\mathcal{J}_{v,n}}|P_mf|^2.
\end{equation}
Combining the inequalities \eqref{Jvn1}, \eqref{Jvn2} with Lemma \ref{lpm}, we get
 \begin{align}\label{last6}
  \Big\|\sum_{v\in I}\sum_{n\in\Gamma_j}|f_{v,n}|^2\Big\|_2&\lesssim\Big\|\sum_{v\in I}\sum_{n\in\Gamma_j}\sum_{m\in\mathcal{J}_{v,n}}|P_mf|^2\Big\|_2\nonumber\\
  &\lesssim\Big\|\sum_{m\in\mathbb{Z}^2}\sum_{\{(v,n)\in I\times \Gamma_j:m\in\mathcal{J}_{v,n}\}}|P_mf|^2\Big\|_2\nonumber\\
  &\lesssim\Big\|\big(\sum_{m\in\mathbb{Z}^2}|P_mf|^2\big)^{1/2}\Big\|_4^2\leq \|f\|_4^2.
 \end{align}
 Hence by the inequalities \eqref{ka2}, \eqref{last5} and \eqref{last6},
 \begin{equation*}
   \Big\|\Big(\sum_{v\in I}\sum_{n\in\Gamma_j}\big|F_v(f_n)\big|^2\Big)^{1/2}\Big\|_{L_4(\mathcal{N})}\leq C_\varepsilon2^{\varepsilon j}\|f\|_4.
 \end{equation*}
By the inequality \eqref{last3}, we have completed the proof of the inequality \eqref{last1}. Thus, the proof of Proposition \ref{key} is now complete.

 \section{Proofs of Lemma \ref{geo} and Lemma \ref{ka}}\label{Se5}
 \subsection{Proof of Lemma \ref{geo}}
 In this subsection, our attention is directed towards proving Lemma \ref{geo}, which involves estimating the number of overlaps between algebraic differences of the sets $\mathcal{U}_{n}^{v}$ and $\mathcal{U}_{m}^{w}$. Recall that the set $\mathcal{U}_{n}^{v}$ is defined as follows:
\begin{equation*}
  \mathcal{U}_{n}^v:=\{(\xi,\tau)\in\mathbb{R}^3:(\xi,\tau)\in\supp\Psi_v,|2^{-j/2}\tau-n|\leq1\}.
\end{equation*}
where $\Psi_v(\xi,\tau)=\psi\big(\frac{|\xi|-\tau}{|\xi|^\epsilon}\big)\chi_v(\xi)$. Here, $\psi$ is a smoothing function on $\mathbb{R}$ with support in $[-2,2]$ and taking value $1$ within  $[-1,1]$. Notably, $\mathcal{U}_{n}^v$ is contained within a larger set (also denoted by $\mathcal{U}_{n}^v$) defined by
\begin{equation*}
 \mathcal{U}_{n}^{v}=\{(\xi,\tau)\in\mathbb{R}^3: dist((\xi,\tau),(\eta,|\eta|))\leq 2^{\epsilon j},\eta\in\Omega_{v,n}\}
\end{equation*}
with $\Omega_{v,n}$ specified as \begin{equation*}
   \Omega_{v,n}:=\{\eta: \ \mathrm{arg}\eta\in [2^{-j/2}(v-1),2^{-j/2}(v+1)], |\eta|\in [2^{j/2}(n-1),2^{j/2}(n+1)]\}.
 \end{equation*}

To establish the desired result, it suffices to prove that there exists a constant $C$, independent of $j$, such that the following inequality holds:
  \begin{equation*}
    \sum_{n,m\in\Gamma_j}\sum_{v,w\in I,|v-w|>1000}\chi_{\mathcal{U}_{n}^v-\mathcal{U}_{m}^{w}}(\xi,\tau)\leq C2^{j/2+2\epsilon j}.
  \end{equation*}

Observe that if $(\xi,\tau)\in \mathcal{U}_n^v$, then $\tau\in[2^{j/2}(n-1),2^{j/2}(n+1)]$. Consequently, for $(\xi,\tau)\in \mathcal{U}_{n}^v-\mathcal{U}_{m}^{w}$, we have $\tau\in[2^{j/2}(n-m-2),2^{j/2}(n-m+2)]$. Given a fixed $n\in\Gamma_j$
  and a fixed point $(\xi,\tau)\in\mathbb{R}^3$, if $(\xi,\tau)\in\mathcal{U}_{n}^{v_1}-\mathcal{U}_{m}^{w_1}$ and $(\xi,\tau)\in\mathcal{U}_{n}^{v_2}-\mathcal{U}_{m^{\prime}}^{w_2}$, we must have $$[2^{j/2}(n-m-2),2^{j/2}(n-m+2)]\cap[2^{j/2}(n-m^{\prime}-2),2^{j/2}(n-m^{\prime}+2)]\neq\emptyset.$$ This implies that $|m-m^{\prime}|\leq4$.
 Hence, for a fixed $(\xi,\tau)$ and $n\in\Gamma_j$, there can be at most nine distinct $m\in\Gamma_j$ such that
  \begin{equation*}
    \sum_{v,w\in I,|v-w|>1000}\chi_{\mathcal{U}_{n}^v-\mathcal{U}_{m}^{w}}(\xi,\tau)\neq0.
  \end{equation*}
Since $|\Gamma_j|\approx 2^{j/2}$,  it remains to demonstrate that for any fixed
  $n,m\in\Gamma_j$, the following estimate holds for an absolute constant $C$,
  \begin{equation}\label{geokey}
    \sum_{v,w\in I,|v-w|>1000}\chi_{\mathcal{U}_{n}^v-\mathcal{U}_{m}^{w}}(\xi,\tau)\leq C2^{2\epsilon j}.
  \end{equation}
  For simplicity,
we can assume that
\begin{equation*}
   \Omega_{v,n}:=\{\eta:\ \mathrm{arg}\eta\in [2^{-j/2}(v-1),2^{-j/2}(v+1)], |\eta|\in [2^{-j/2}(n-1),2^{-j/2}(n+1)]\},
 \end{equation*}
 and
 \begin{equation*}
 \mathcal{U}_{n}^{v}=\{(\xi,\tau)\in\mathbb{R}^3: dist((\xi,\tau),(\eta,|\eta|))\leq 2^{\epsilon j-j},\eta\in\Omega_{v,n}\}.
\end{equation*}
Elements in $\mathcal{U}_{n}^{v}$ can be expressed as
 \begin{equation*}
   (2^{-j/2}n+r_v)\cdot(e^{i(2^{-jv/2+\theta})},1)+\zeta_v,
 \end{equation*}
  where $|r_v|\leq 2^{-j/2},|\theta|\leq 2^{-j/2}$ and $\zeta_v$ is a  vector in $\mathbb{R}^3$ with $|\zeta_v|\leq2^{\epsilon j-j}$.
  For the sake of simplicity in the following proof, we denote $2^{-j/2}$ by $\delta$. Without loss of generality, we assume $v>w$, $\delta n=1$, $\delta m=\lambda$ for some constant $\lambda\thickapprox1$. Elements of $\mathcal{U}_{n}^v-\mathcal{U}_{m}^{w}$ have the form
 \begin{equation}\label{form}
   ((1+r_v)\cdot e^{i({v\delta+\theta_v})},r_v)-((\lambda+r_w)\cdot e^{i({w\delta+\theta_w})},r_w)+\zeta_{v}-\zeta_w,
\end{equation}
 where $|r_v|,|r_w|,|\theta_v|,|\theta_w|\leq \delta$ and $|\zeta_{v}|,|\zeta_w|\leq\delta^{2-2\epsilon}$.
  To establish the inequality \eqref{geokey}, it suffices to demonstrate that for a fixed pair $(v_1,w_1)$,
   \begin{equation}\label{geokey1}
     |\{(v_2,w_2): \mathcal{U}_{n}^{v_1}-\mathcal{U}_{n^\prime}^{w_1}\cap \mathcal{U}_{n}^{v_2}-\mathcal{U}_{n^\prime}^{w_2}\neq\emptyset\}\lesssim\delta^{-4\epsilon}.
   \end{equation}
   In fact, we will show below that $|v_1-v_2|,|w_1-w_2|\lesssim\delta^{-2\epsilon}$, which immediately implies the inequality \eqref{geokey1}.

 Given that $\mathcal{U}_{n}^{v_1}-\mathcal{U}_{m}^{w_1}\cap \mathcal{U}_{n}^{v_2}-\mathcal{U}_{m}^{w_2}\neq\emptyset$, by \eqref{form}, there exist $|r_{v_i}|,|r_{w_i}|, |\theta_{v_i}|,|\theta_{w_i}|\leq\delta$ for $i=1,2,$ and $\zeta\in\mathbb{R}^3$ with $|\zeta|\leq4\delta^{2-2\epsilon}$ such that
  \begin{align*}
     &((1+r_{v_1})\cdot e^{i({v_1\delta+\theta_{v_1}})},1+r_{v_1})-((\lambda+r_{w_1})\cdot e^{i({{w_1}\delta+\theta_{w_1}})},\lambda+r_{w_1})+\zeta  \\
    & =((1+r_{v_2})\cdot e^{i({{v_2}\delta+\theta_{v_2}})},1+r_{v_2})-((\lambda+r_{w_2})\cdot e^{i({{w_2}\delta+\theta_{w_2}})},\lambda+r_{w_2}).
  \end{align*}
  Comparing the third coordinate, we obtain
  \begin{equation}\label{geol1}
    |r_{v_1}-r_{v_2}-r_{w_1}+r_{w_2}|\leq 4\delta^{2-2\epsilon}.
  \end{equation}
  Next, by comparing the first two coordinates, we have
  \begin{align*}
     &(1+r_{v_1})( e^{i({v_1\delta+\theta_{v_1}})}-e^{i({{v_2}\delta+\theta_{v_2}})})+(r_{v_1}-r_{v_2})(e^{i({{v_2}\delta+\theta_{v_2}})}-e^{i({{w_2}\delta+\theta_{w_2}})})\nonumber\\
    &=(\lambda+r_{w_1})(e^{i({{w_1}\delta+\theta_{w_1}})}-e^{i({{w_2}\delta+\theta_{w_2}})})+(r_{w_1}-r_{w_2}-r_{v_1}+r_{v_2})e^{i({{w_2}\delta+\theta_{w_2}})}+x,
  \end{align*}
   where $x\in\mathbb{R}^2$ with $|x|\leq 4\delta^{2-2\epsilon}$, by the inequality \eqref{geol1}, we have
  \begin{align}\label{geol2}
    &(1+r_{v_1})( e^{i({v_1\delta+\theta_{v_1}})}-e^{i({{v_2}\delta+\theta_{v_2}})})+(r_{v_1}-r_{v_2})(e^{i({{v_2}\delta+\theta_{v_2}})}-e^{i({{w_2}\delta+\theta_{w_2}})})\nonumber\\
    &=(\lambda+r_{w_1})(e^{i({{w_1}\delta+\theta_{w_1}})}-e^{i({{w_2}\delta+\theta_{w_2}})})+\tilde{x},
  \end{align}
  where $\tilde{x}\in\mathbb{R}^2$ with $|\tilde{x}|\leq|\zeta|+|r_{v_1}-r_{v_2}-r_{w_1}+r_{w_2}|\leq 8\delta^{2-2\epsilon}$.
  Using the identity
  \begin{equation*}
    e^{i\alpha_1}-e^{i\alpha_2}=2\sin((\alpha_1-\alpha_2)/2)ie^{i((\alpha_1+\alpha_2)/2)},
 \end{equation*}
the left-hand side of equation  \eqref{geol2} becomes
   \begin{align}\label{geol3}
    &2(1+r_{v_1})\sin\big(\big((v_1-v_2)\delta+\theta_{v_1}-\theta_{v_2}\big)/2\big)ie^{i(((v_1+v_2)\delta+\theta_{v_1}+\theta_{v_2})/2)}\nonumber\\
    &+2(r_{v_1}-r_{v_2})\sin\big(\big((v_2-w_2)\delta+\theta_{v_2}-\theta_{w_2}\big)/2\big)ie^{i(((v_2+w_2)\delta+\theta_{v_2}+\theta_{w_2})/2)},
\end{align}
while the right-hand side of equation \eqref{geol2} becomes
\begin{equation}\label{geol4}
  2(1+r_{w_1})\sin\big(\big((w_1-w_2)\delta+\theta_{w_1}-\theta_{w_2}\big)/2\big)ie^{i(((w_1+w_2)\delta+\theta_{w_1}+\theta_{w_2})/2)}+\tilde{x}.
\end{equation}
The length of the projection from \eqref{geol3} onto the direction $e^{i(((w_1+w_2)\delta+\theta_{w_1}+\theta_{w_2})/2)}$  is denoted by $D$, and it equals
\begin{align} \label{geol5}
 &\big|2(1+r_{v_1})\sin\big(((w_1+w_2-v_1-v_2)\delta+\theta_{w_1}+\theta_{w_2}-\theta_{v_1}-\theta_{v_2})/2\big)\nonumber\\
& \times \sin\big(\big((v_1-v_2)\delta+\theta_{v_1}-\theta_{v_2}\big)/2\big)+2(r_{v_1}-r_{v_2})\sin\big(\big((v_2-w_2)\delta+\theta_{v_2}-\theta_{w_2}\big)/2\big)\nonumber\\
& \times\sin\big(((w_1-v_2)\delta+\theta_{w_1}-\theta_{v_2})/2\big)\big|.
\end{align}
Recalling that  $\frac{\alpha}{\pi}\leq\sin(\alpha)\leq\alpha$ for $\alpha\in[0,\frac{\pi}{2}]$, $|r_{v_i}|, |r_{w_i}|,|\theta_{v_i}|, |\theta_{w_i}|\leq \delta$ and $v_1-w_1,v_2-w_2>1000$, one can derive
\begin{equation}\label{keyy1}
  D\geq \frac{|(v_1-v_2)\delta+\theta_{v_1}-\theta_{v_2}|}{10}(v_2-w_2+v_1-w_1-4)\delta-(v_2-w_2)\delta^2.
\end{equation}
On the other hand, \eqref{geol4} implies $D\leq|\tilde{x}|\leq8 \delta^{2-2\epsilon}$. We claim that $|v_1-v_2|\leq 100 \delta^{-2\epsilon}$,
otherwise, by the inequality \eqref{keyy1},
\begin{equation*}
 8 \delta^{2-2\epsilon}\geq 2(v_2-w_2)\delta^{2-2\epsilon}-(v_2-w_2)\delta^2\geq (v_2-w_2)\delta^{2-2\epsilon},
\end{equation*}
leading to a contradiction since $v_2-w_2>1000$.
Using the same trick when we consider the projection onto the direction $ie^{i(((w_1+w_2)\delta+\theta_{w_1}+\theta_{w_2})/2)}$, we obtain
\begin{align*}
  &\big|2(1+r_{w_1})\sin\big(\big((w_1-w_2)\delta+\theta_{w_1}-\theta_{w_2}\big)/2\big)\big|\\&\leq|\tilde{x}|+\big|2(1+r_{v_1})\sin\big(\big((v_1-v_2)\delta+\theta_{v_1}-\theta_{v_2}\big)/2\big)\big|+|r_{v_1}-r_{v_2}|\\
 &\leq 8 \delta^{2-2\epsilon}+100\delta^{1-2\epsilon}+2\delta\lesssim\delta^{1-2\epsilon}.
\end{align*}
This implies
\begin{equation*}
  |(w_1-w_2)\delta+\theta_{w_1}-\theta_{w_2}|\lesssim\delta^{1-2\epsilon},
\end{equation*}
and thus $|w_1-w_2|\lesssim\delta^{-2\epsilon}$ since $|\theta_{w_1}-\theta_{w_2}|\leq2\delta$.
Hence, we conclude the inequality \eqref{geokey1}, which leads to the desired lemma.
\begin{remark}\label{rem5}
 \rm{In the paper by Mockenhaupt, Seeger, and Sogge \cite{MSS}, they derived an estimate for the number of overlaps of algebraic sums, given by:
  \begin{equation*}
   \sum_{v,w\in I}\chi_{\mathcal{U}_{n}^v+\mathcal{U}_{m}^{w}}(\xi,\tau)\leq Cj2^{j/2+2\epsilon j}.
  \end{equation*}
  Our method is equally applicable in this context, and furthermore, allows us to refine the estimate slightly. Specifically, we obtain:
   \begin{equation*}
   \sum_{v,w\in I}\chi_{\mathcal{U}_{n}^v+\mathcal{U}_{m}^{w}}(\xi,\tau)\leq C2^{j/2+2\epsilon j}.
  \end{equation*}}

  \end{remark}

 \subsection{Proof of Lemma \ref{ka}}
Let us begin by estimating the kernel $K_{v,n}(y,t)$. Recalling the definitions of $F_v(f_n)$ from \eqref{fn} and \eqref{Fv}, we derive
\begin{equation*}
  K_{v,n}(y,t)=\rho_1(t)\int_{\mathbb{R}^2}e^{i(y\cdot\xi+t|\xi|)}\rho_0(|2^{-j}\xi|)\chi_v(\xi)a(\xi,t)\tilde{\phi}(\xi)\ d\xi,
\end{equation*}
where\begin{equation*}
  \tilde{\phi}(\xi)=\sum_{\{m\in\mathbb{Z}^2:\big||m|-n\big|<100\}}\rho(2^{-j}|\xi|)\phi(2^{-j/2}\xi_1-m_1)\phi(2^{-j/2}\xi_2-m_2).
\end{equation*}
Through integration by parts, one can establish
 \begin{equation}\label{keyes}
  |K_{v,n}(y,t)|\leq C_N\frac{2^j}{(1+|2^j(\langle y,\xi_v\rangle+t)|^2)^N}\frac{2^{j/2}}{(1+|2^{j/2}(y-\langle y,\xi_v\rangle\xi_v)|^2)^N}
 \end{equation}
for any $N>0$ (see e.g. \cite[(1.10)]{MSS} for further details). Utilizing the inequality \eqref{keyes}, we deduce the inequality \eqref{ka1}, which states
 \begin{equation*}
    \sup_{y,t,v,n}\int_{\mathbb{R}^2}|K_{v,n}(y-x,t)|dx\lesssim1.
   \end{equation*}
  Next, it suffices to prove the inequality \eqref{ka2}, which asserts that for  $g\in L_2(\mathcal{N})$,  we have
   \begin{equation}\label{kakeya1}
     \Big\|{\sup_{v\in I, n\in\Gamma_j}}^+\int_{\mathbb{R}^3}|K_{v,n}(y,t;x)|g(y,t)dydt\Big\|_2\leq C_\varepsilon2^{\varepsilon j}\|g\|_{L_2(\mathcal{N})}
   \end{equation}
   for $\varepsilon>0$.

  Using a standard argument, it suffices to demonstrate the inequality \eqref{kakeya1} for $g\in \mathcal{S}_+(\mathbb{R}^3)\otimes \mathcal{S}_+(\mathcal{M})$.
 Given the inequality \eqref{keyes}, we observe that  $K_{v,n}(y,t)$ is essentially supported within a rectangle of size $2^{-j}\times 2^{-j/2}\times1$ centered around the ray $\gamma_{v}$, defined as:
 \begin{equation*}
   \gamma_{v}:=\{(y,t)\in\mathbb{R}^3:y+t\xi_v=0\}.
 \end{equation*}
 Motivated by this observation, for $\ell\geq0$,  we define the rectangle of size $2^{-j+\ell}\times 2^{-j/2+\ell}\times1$ centered around $\gamma_{v}$ as
 \begin{equation*}
   \mathcal{R}_{v,\ell}:=\{(y,0)+t(-\xi_v,1)\in\mathbb{R}^2\times[1,2]:|\langle y, \xi_v\rangle|\leq 2^{-j+\ell},\ |\langle y, \xi_v^\bot\rangle|\leq 2^{-j/2+\ell}\},
 \end{equation*}
 where $\xi_v^\bot$ is the unit vector orthogonal to $\xi_v$, and $\mathcal{R}_{v,-1}:=\varnothing$.
Utilizing the inequality \eqref{keyes} and noting that $\supp\rho_1\subset[1,2]$, we can write
 \begin{equation}\label{e:dyakakeya}
 \begin{split}
  \int_{\mathbb{R}^3}|K_{v,n}(y-x,t)|g(y,t)dydt&=\sum_{\ell=0}^\infty \int_{\mathcal{R}_{v,\ell} \backslash\mathcal{R}_{v,\ell-1}}|K_{v,n}(y,t)|g(x+y,t)dydt\\
  &\leq C_N\sum_{\ell=0}^\infty  \int_{\mathcal{R}_{v,\ell} \backslash\mathcal{R}_{v,\ell-1}}\frac{2^{3j/2}g(x+y,t)dydt}{(1+2^{2\ell})^N(1+2^{2\ell})^N}\\
  &\leq C_N\sum_{\ell=0}^\infty 2^{-2\ell(2N-1)} \frac{1}{|\mathcal{R}_{v,\ell}|}\int_{\mathcal{R}_{v,\ell} }g(x+y,t)dydt.
 \end{split}
 \end{equation}
 Furthermore, we split each $\mathcal{R}_{v,\ell}$ along its longer dimension of size $2^{-j/2+\ell}$ into $2^{j/2+\ell}$ pieces and along its shorter dimension of size $2^{-j+\ell}$ into $2^{\ell}$ pieces, yielding a collection of rectangles $\{\mathcal{R}_{v,\ell,i}\}_{i=1}^{{i(\ell)}}$ of dimensions $2^{-j}\times 2^{-j}\times1$, where $i(\ell)\thickapprox 2^{j/2+2\ell}$. Specifically
\begin{equation*}
   \mathcal{R}_{v,\ell,i}:=\{(y,0)+(x_i,0)+t(-\xi_v,1)\in\mathbb{R}^2\times[1,2]:|\langle y, \xi_v\rangle|\leq 2^{-j},\ |\langle y ,\xi_v^\bot\rangle|\leq 2^{-j}\},
 \end{equation*}
 for some $x_i\in\mathbb{R}^2$ such that $\mathcal{R}_{v,\ell}\subset\bigcup_{i=0}^{i(\ell)}\mathcal{R}_{v,\ell,i}$. Let $\gamma_{v,x_i}$ denote the ray $\gamma_{v}+(x_i,0)$ and $\mathcal{T}_{v,i}$ the cylinder defined as
 \begin{equation*}
   \mathcal{T}_{v,i}:=\{ (y,t)\in\mathbb{R}^2\times[1,2]: dist\{(y,t),\gamma_{v,x_i}\}<2^{-j+1}\}.
 \end{equation*}
 It is straightforward to verify that $\mathcal{R}_{v,\ell,i}\subset \mathcal{T}_{v,i}$. Note that the measures $|\mathcal{R}_{v,\ell}|\thicksim 2^{\frac{-3j}{2}+2\ell}$ and $|\mathcal{T}_{v,i}|\thicksim 2^{-2j+2}$. Proceeding with the inequalities \eqref{e:dyakakeya}, we have
 \begin{align*}
  \int_{\mathbb{R}^3}|K_{v,n}(y-x,t)|g(y,t)dydt
   &\leq C_N\sum_{\ell=0}^\infty 2^{-2\ell(2N-1)} \frac{1}{|\mathcal{R}_{v,\ell}|}\int_{\mathcal{R}_{v,\ell} }g(x+y,t)dydt\\
   &\leq C_N\sum_{\ell=0}^\infty\sum_{i=0}^{i(\ell)} 2^{-2\ell(2N-1)} \frac{1}{|\mathcal{R}_{v,\ell}|}\int_{\mathcal{R}_{v,\ell,i}}g(x+y,t)dydt\\
  &\leq C_N\sum_{\ell=0}^\infty\sum_{i=0}^{i(\ell)} 2^{-4\ell N-j/2} \frac{1}{|\mathcal{T}_{v,i}|}\int_{\mathcal{T}_{v,i}}g(x+y,t)dydt.
 \end{align*}
 Therefore, we deduce that
 \begin{align*}
   &\Big\|{\sup_{v\in I, n\in\Gamma_j}}^+\int_{\mathbb{R}^3}|K_{v,n}(y-x,t)|g(y,t)dydt\Big\|_2\\&\leq C_N\sum_{\ell=0}^\infty\sum_{i=0}^{i(\ell)} 2^{-4\ell N-j/2} \Big\|{\sup_{v\in I}}^+\frac{1}{|\mathcal{T}_{v,i}|}\int_{\mathcal{T}_{v,i}}g(x+y,t)dydt\Big\|_2.
 \end{align*}
Hence, it suffices to establish
\begin{align}\label{kakeyaf}
  \Big\|{\sup_{v\in I}}^+\frac{1}{|\mathcal{T}_{v,i}|}\int_{\mathcal{T}_{v,i}}g(x+y,t)dydt\Big\|_2\leq C_\varepsilon2^{\varepsilon j}\|g\|_{L_2(\mathcal{N})}.
\end{align}

 To prove the estimate \eqref{kakeyaf}, we require an operator-valued Kakeya-type estimate.
 For $\theta\in[0,2\pi]$ and $0<\delta<1/2$, we define a ray as
 \begin{equation*}
   \gamma_\theta:=\{(y,t)\in\mathbb{R}^3:y+t(\cos\theta,\sin\theta)=0\}
 \end{equation*}
 and a unit cylinder as
 \begin{equation*}
   \mathcal{R}_\theta:=\{ (y,t)\in\mathbb{R}^2\times[0,1]: dist\{(y,t),\gamma_\theta\}<\delta\}.
 \end{equation*}
 \begin{lemma}\label{kakeya} For $g\in \mathcal{S}_+(\mathbb{R}^3)\otimes \mathcal{S}_+(\mathcal{M})$, we have
    \begin{equation}\label{kakeya2}
     \Big\|{\sup_{\theta\in[0,2\pi]}}^+\frac{1}{\mathcal{R}_\theta}\int_{\mathcal{R}_\theta}g(x+y,t)dydt\Big\|_2\leq C|\log_2\delta|^{3}\|g\|_{L_2(\mathcal{N})}.
   \end{equation}
 \end{lemma}
 \begin{proof}
 Firstly, we select a suitable function $a\in C_0^\infty(\mathbb{R}^2)$  that satisfies $\check{a}\geq0$, and define
   $a_\delta(t,\xi):=\chi_{[0,1]}(t)a(\delta\xi)$. Note that
   \begin{align*}
    dist\{(y,t),\gamma_\theta\}^2&= |y|^2+t^2/2+t\langle y, (\cos\theta,\sin\theta)\rangle-\langle y, (\cos\theta,\sin\theta)\rangle^2/2\\
     &\geq |y+t(\cos\theta,\sin\theta)|^2/2.
   \end{align*}
Then, for $(y,t)\in\mathcal{R}_\theta$, we have
 $\check{a}_\delta( y+t(\cos\theta,\sin\theta))\gtrsim1/\delta^2$.  Therefore, we can bound the integral as follows:
 \begin{align*}
   \frac{1}{\mathcal{R}_\theta}\int_{\mathcal{R}_\theta}g(x+y,t)dydt&\lesssim\int_{\mathbb{R}^3} \check{a}_\delta(y+t(\cos\theta,\sin\theta))g(x+y,t)dydt\\
   &=\int_{\mathbb{R}^3} \check{a}_\delta(-x+y+t(\cos\theta,\sin\theta))g(y,t)dydt.
 \end{align*}
  It suffices to show
   \begin{equation}\label{kakeya3}
     \Big\|{\sup_{\theta\in[0,2\pi]}}^+A_\theta g(x)\Big\|_2\leq C|\log_2\delta|^{3}\|g\|_{L_2(\mathcal{N})},
   \end{equation}
   where
   \begin{equation*}
     A_\theta g(x)=\int_{\mathbb{R}}\int_{\mathbb{R}^2}e^{i[-\langle x,\xi\rangle+t\langle(\cos\theta,\sin\theta),\xi\rangle]} a_\delta(t,\xi)\tilde{g}(\xi,t)d\xi dt
   \end{equation*}
   and
   \begin{equation*}
    \tilde{g}(\xi,t)=\int_{\mathbb{R}^2}e^{iy\xi}g(y,t)dy.
   \end{equation*}

   To prove the inequality \eqref{kakeya3}, we define dyadic operators as follows:
   \begin{equation*}
     A_\theta^\lambda g(x)=\int_{\mathbb{R}}\int_{\mathbb{R}^2}e^{i[-\langle x,\xi\rangle+t\langle(\cos\theta,\sin\theta),\xi\rangle]} a_\delta(t,\xi)\beta(|\xi|/\lambda)\tilde{g}(\xi,t)d\xi dt.
   \end{equation*}
   Here $\beta$ is a smoothing function with support in $[1/2,2]$ and satisfies $\sum_{j\in\mathbb{Z}}\beta(2^{-j}r)=1$ for all $r>0$.
   Note that $a\in C_0^\infty(\mathbb{R}^2)$, one can decompose $A_\theta$ as $$A_\theta=\sum_{1<l<\log_2{\delta^{-1}}+C}A_\theta^{2^l}+\tilde{R}_\theta,$$ where $C$ is a fixed constant, and the kernel of
$\tilde{R}_\theta$ is controlled by $O(1+|y-x|)^{-N}$ for any $N$ with bounds independent of $\theta$ (see \cite[Page 89]{Sogge17}). Therefore, it suffices to demonstrate that
   \begin{equation}\label{kakeya4}
     \Big\|{\sup_{\theta\in[0,2\pi]}}^+ A_\theta^\lambda g(x)\Big\|_2\leq C(\log_2\lambda)^2\|g\|_{L_2(\mathcal{N})},\ \lambda>2.
   \end{equation}
   For $k=1,2,...$, we define
   \begin{equation*}
      A_\theta^{\lambda,k} g(x)=\int_{\mathbb{R}}\int_{\mathbb{R}^2}e^{i[-\langle x,\xi\rangle+t\langle(\cos\theta,\sin\theta),\xi\rangle]} a_\delta(t,\xi)\beta(|\xi|/\lambda)\beta_{\lambda,k}(\xi,\theta)\tilde{g}(\xi,t)d\xi dt,
   \end{equation*}
   where $\beta_{\lambda,k}(\xi,\theta)=\beta\big(2^{-k}\lambda^{1/2}\langle(-\sin\theta,\cos\theta),\frac{\xi}{|\xi|}\rangle\big).$
   Note that $|\langle(-\sin\theta,\cos\theta),\frac{\xi}{|\xi|}\rangle|\leq 1$ and the support of $\beta$ is contained in $[1/2,2]$, $A_\theta^{\lambda,k}$ is nonzero only if $2^{k-1}\leq\lambda^{1/2}$. Furthermore, we define $$A_\theta^{\lambda,0}=A_\theta^{\lambda}-\sum_{0\leq k-1\leq\log_2\lambda^{1/2}}A_\theta^{\lambda,k}.$$
   Thus, it suffices to show that for all $k=0,1,2,...$, there exists a constant $C$ such that following estimate holds:
   \begin{equation}\label{kakeya5}
     \Big\|{\sup_{\theta\in[0,2\pi]}}^+A_\theta^{\lambda,k} g(x)\Big\|_2\leq C\|g\|_{L_2(\mathcal{N})}.
   \end{equation}
   Since $A_\theta^\lambda g$ is positive, by Lemma \ref{max interpolation}, we have
   \begin{equation}\label{kare1}
     \Big\|{\sup_{\theta\in[0,2\pi]}}^+ A_\theta^{\lambda,k} g(x)\Big\|_2^2\leq \Big\|\sup_{\theta\in[0,2\pi]} |A_\theta^{\lambda,k} g(x)|^2\Big\|_1.
   \end{equation}
 Applying Lemma \ref{remax}, for all $c\leq2\pi$, we obtain
   \begin{align}\label{kare2}
    \Big\|{\sup_{r\in[0,2\pi]}}^+ |A_r^{\lambda,k} g(x)|^2\Big\|_1&\leq 2c^{-1}\Big\|\int_{0}^{2\pi}|A_\theta^{\lambda,k} g |^2d\theta\Big\|_1+2c
    \Big\|\int_{0}^{2\pi}\big|\frac{\partial A_\theta^{\lambda,k} g}{\partial\theta}\big|^2d\theta\Big\|_1.
   \end{align}
  Letting $c=\lambda^{-1/2}2^{-k}$, combining the inequalities \eqref{kare1} and \eqref{kare2}, to prove \eqref{kakeya5},
 we only need to show
  \begin{equation}\label{kare5}
    \Big\|\int_0^{2\pi}\Big|(\frac{\partial}{\partial\theta})^\ell A_\theta^{\lambda,k} g\Big|^2d\theta\Big\|_1^{1/2}\lesssim(\lambda^{-1/4}2^{-k/2})^{1-2\ell}\|g\|_{L_2(\mathcal{N})},\ \ell=0,1.
  \end{equation}
  Recall that on the support of the symbol of $A_\theta^{\lambda,k}$, we have $\big|\langle(-\sin\theta,\cos\theta),\frac{\xi}{|\xi|}\rangle\big|\thickapprox 2^k\lambda^{-1/2}$ and $|\xi|\thickapprox\lambda$. Therefore,
  \begin{equation*}
    \Big|\Big(\frac{\partial}{\partial\theta}\Big)\langle(\cos\theta,\sin\theta),\xi\rangle\Big|=\big|\langle(-\sin\theta,\cos\theta),\xi\rangle\big|\thickapprox 2^k\lambda^{1/2},
  \end{equation*}
  and $\Big|\Big(\frac{\partial}{\partial\theta}\Big)\beta_{\lambda,k}(\xi,\theta)\Big|\lesssim 2^{-k}\lambda^{1/2}$. Consequently, $\big(\frac{\partial }{\partial\theta}\big)A_\theta^{\lambda,k}$ behaves like $2^k\lambda^{1/2}A_\theta^{\lambda,k}$, it suffices to prove the inequality \eqref{kare5} for $\ell=0$.

Recall that
\begin{align*}
   A_\theta^{\lambda,k} g(x)&=\int_{\mathbb{R}}\int_{\mathbb{R}^2}e^{i[-\langle x,\xi\rangle+t\langle(\cos\theta,\sin\theta),\xi\rangle]} a_\delta(t,\xi)\beta(|\xi|/\lambda)\beta_{\lambda,k}(\xi,\theta)\tilde{g}(\xi,t)d\xi dt\\
   &=\int_{\mathbb{R}^2}e^{-i\langle x,\xi\rangle}\int_{\mathbb{R}}e^{it\langle(\cos\theta,\sin\theta),\xi\rangle} a_\delta(t,\xi)\beta(|\xi|/\lambda)\beta_{\lambda,k}(\xi,\theta)\tilde{g}(\xi,t)dtd\xi.
\end{align*}
By the Plancherel theorem, we can derive the following expression:
\begin{align*}
    &\Big\|\int_0^{2\pi}\Big| A_\theta^{\lambda,k}g\Big|^2d\theta\Big\|_1 \\&=\tau\Big(\int_{\mathbb{R}^2}\int_0^{2\pi}\Big| A_\theta^{\lambda,k} g(x)\Big|^2d\theta dx\Big)\\&
    =\tau\Big(\int_{\mathbb{R}^2}\int_0^{2\pi}\Big|\int_{\mathbb{R}}e^{it\langle(\cos\theta,\sin\theta),\xi\rangle} a_\delta(t,\xi)\beta(|\xi|/\lambda)\beta_{\lambda,k}(\xi,\theta)\tilde{g}(\xi,t)dt\Big|^2d\theta d\xi\Big)\\& \\
     &=\tau\Big(\int_{\mathbb{R}^2}\int_0^{1}\int_0^1H^{\lambda,k}(t-t^\prime,\xi)|\beta(|\xi|/\lambda)a(\delta\xi)|^2\tilde{g}^*(\xi,t)\tilde{g}(\xi,t^\prime)dtdt^\prime d\xi\Big), \\
  \end{align*}
where
  \begin{equation*}
    H^{\lambda,k}(t-t^\prime,\xi)=\int_0^{2\pi}e^{i(t-t^\prime)\langle(\cos\theta,\sin\theta),\xi\rangle}|\beta_{\lambda,k}(\xi,\theta)|^2d\theta.
  \end{equation*}
For $k\geq0$ and any $N$, it is known (see \cite{MSS}  for further details) that
  \begin{equation*}
   | H^{\lambda,k}(t-t^\prime,\xi)|\leq C_N\lambda^{-1/2}2^k(1+2^{2k}|t-t^\prime|)^{-N},\ |\xi|\thickapprox\lambda.
  \end{equation*}
Combining these results with the H\"{o}lder inequality and the Young inequality, we obtain
\begin{align*}
 &\tau\Big(\int_{\mathbb{R}^2}\int_0^{1}\int_0^1H^{\lambda,k}(t-t^\prime,\xi)|\beta(|\xi|/\lambda)a(\delta\xi)|^2\tilde{g}^*(\xi,t)\tilde{g}(\xi,t^\prime)dtdt^\prime d\xi\Big)\\
 &\leq C_N\lambda^{-1/2}2^k \tau\Big(\int_{\mathbb{R}^2}\int_0^{1}\int_0^1(1+2^{2k}|t-t^\prime|)^{-N}|\beta(|\xi|/\lambda)a(\delta\xi)|^2|\tilde{g}^*(\xi,t)\tilde{g}(\xi,t^\prime)|dtdt^\prime d\xi\Big)\\
 &=C_N\lambda^{-1/2}2^k \int_{\mathbb{R}^2}\tau\Big(\int_0^{1}\int_0^1(1+2^{2k}|t-t^\prime|)^{-N}|\tilde{g}^*(\xi,t)\tilde{g}(\xi,t^\prime)|dtdt^\prime \Big)|\beta(|\xi|/\lambda)a(\delta\xi)|^2d\xi\\
 &\leq C_N\lambda^{-1/2}2^k \int_{\mathbb{R}^2}\tau\Big(\int_0^{1}\big|(1+2^{2k}|\cdot|)^{-N}\ast|\tilde{g}(\xi,\cdot)|\big|^2dt \Big)^{1/2}\\&\quad\times\tau\Big(\int_0^{1}|\tilde{g}(\xi,t)|^2dt\Big)^{1/2}|\beta(|\xi|/\lambda)a(\delta\xi)|^2d\xi\\
 &\leq C_N\lambda^{-1/2}2^{-k}\int_{\mathbb{R}^2}\tau\Big(\int_0^{1}|\tilde{g}(\xi,t)|^2dt\Big)|\beta(|\xi|/\lambda)a(\delta\xi)|^2d\xi\leq C_N\lambda^{-1/2}2^{-k}\|g\|_{L_2(\mathcal{N})}^2.
\end{align*}
This completes the proof of the inequality \eqref{kare5}.
 \end{proof}
By applying Lemma \ref{kakeya} and the translation transform, we derive the inequality \eqref{kakeyaf}. Thus, the proof of Lemma \ref{ka} is concluded.

\section{Proof of Theorem \ref{qeloc}}\label{S6}
In this section, we provide the proof of Theorem \ref{qeloc}. To start, we first introduce the transference technique by defining a normal injective
 $*$-homomorphism, denoted as $\sigma_\theta$, from $\mathcal{R}_\theta^2$ to $L_\infty(\mathbb{R}^2)\overline{\otimes}\mathcal{R}_\theta^2$. This map is specified by $\sigma_\theta (U(\xi)):=\exp_\xi\otimes U(\xi)$, where $\exp_\xi$ represents the character $x \rightarrow \exp( i(x,\xi))$ in $L_\infty(\mathbb{R}^2)$ (see \cite[Corollary 1.4]{GJP} for further details). Let $m$ be a reasonable function, and $T_m$ denote the associated Fourier multiplier on $L_\infty(\mathbb{R}^2)$. And we denote the associated Fourier multiplier on $\mathcal{R}_\theta^2$ still by $T_m$ by slightly abusing the notation. Then, we have the following intertwining identity:
 \beq\label{ii}
 \sigma_\theta\circ T_m=(T_m\otimes id_{\mathcal{R}_\theta^2})\circ \sigma_\theta.\eeq

 Subsequently, we present the family of $*$-automorphism $(\alpha_\eta)_{\eta\in\mathbb{R}^2}$ on $\mathcal{R}_\theta^2$, defined by $\alpha_\eta(U(\xi)):=\exp(i(\xi,\eta))U(\xi)$ for each $\eta\in\mathbb{R}^2$. These automorphisms satisfy several crucial properties:
  \begin{itemize}
    \item[(1)] For all $x \in\mathcal{R}_\theta^2$, the map $\eta\rightarrow \alpha_\eta x$ from $\mathbb{R}^2$  to $\mathcal{R}_\theta^2$ is weak*- continuous ;
    \item[(2)] For all $\eta\in\mathbb{R}^2$ , $\alpha_\eta$ is a $*$-automorphism of $\mathcal{R}_\theta^2$;
    \item[(3)] For all $\eta\in\mathbb{R}^2$, we have $\tau_\theta=\tau_\theta \circ \alpha_\eta$.
  \end{itemize}
 Combining these properties with Lemma 1.1 from \cite{JX07}, we deduce the following proposition.
 \begin{proposition}\label{iso}
 Let $x \in L_p(\mathcal{R}_\theta^2)$ with $1\leq p<\infty$. Then the map $\eta\rightarrow \alpha_\eta x$ from $\mathbb{R}^2$ to $L_p(\mathcal{R}_\theta^2)$ is continuous in the norm topology, and for $\eta\in\mathbb{R}^2$, $\alpha_\eta$ is an isometry on $L_p(\mathcal{R}_\theta^2)$.
   \end{proposition}

Given that $\mathbb{R}^2$ equipped with the Lebesgue measure $d\eta$ is not a probability space, the intertwining identity \eqref{ii} can not be efficiently applied.  An approximation argument is required. We select the heat kernels $h_\varepsilon(\eta)=(\varepsilon/\pi)e^{-\varepsilon|\eta|^2}$. Equipped with the Gaussian measure $h_\varepsilon(\eta)d\eta$, $\mathbb{R}^2$ becomes a probability space. We denote $ L_\infty(\mathbb{R}^2, d\eta)\overline{\otimes}\mathcal{R}_\theta^2$ by $\mathcal{N}_\theta$.

\begin{lemma}\label{lem1}
Let $x\in L_p(\mathcal{R}_\theta^2)$ with $1\leq p\leq \infty$. We have
  for any $\varepsilon>0$,
  \begin{align}\label{tral}
 \|x\|_{L_p(\mathcal{R}_\theta^2)}=\|h_\varepsilon^{1/p}\sigma_\theta(x)\|_{L_p(\mathcal{N_\theta})}.
  \end{align}
\end{lemma}

\begin{proof}
Applying Proposition \ref{iso}, we have
  \begin{align*}
    \|x\|_{L_p(\mathcal{R}_\theta^2)}^p&=\int_{\mathbb{R}^d}h_\varepsilon(\eta)\|x\|_{L_p(\mathcal{R}_\theta^2)}^pd\eta=\int_{\mathbb{R}^2}h_\varepsilon(\eta)\|\alpha_\eta x\|_{L_p(\mathcal{R}_\theta^2)}^pd\eta\\
    &=\int_{\mathbb{R}^2}h_\varepsilon(\eta)\|\sigma_\theta (x)(\eta)\|_{L_p(\mathcal{R}_\theta^2)}^pd\eta=\|h_\varepsilon^{1/p}\sigma_\theta (x)\|_{L_p(\mathcal{N}_\theta)}^p.
  \end{align*}
\end{proof}
We are now at the position to show Theorem \ref{qeloc}.
Given $x_0=U_\theta(f_0),x_1=U_\theta(f_1)$ for some $f_0,f_1\in\mathcal{S}(\mathbb{R}^2)$, and under the condition $\nu>\frac{1}{2}-\frac{1}{p}-\kappa(p)$, we define
\begin{equation}\label{nu}
m_{0,t}^\nu := c_0 \frac{\cos(|t\xi|)}{(1+|\xi|^2)^{\nu/2}} \quad \text{and} \quad m_{1,t}^\nu := c_1 \frac{\sin(|t\xi|)}{|\xi|(1+|\xi|^2)^{(\nu-1)/2}}.
\end{equation}
We then set
\begin{equation}
u_\nu(t) = T_{m_{0,t}^\nu}(x_0) + T_{m_{1,t}^\nu}(x_1).
\end{equation}
To establish Theorem \ref{qeloc}, it is equivalent to show
\begin{equation}\label{lasteq}
  \|u_\nu(t)\|_{L_p(L_\infty([1,2])\overline{\otimes}\mathcal{R}_\theta^2))}\lesssim\|x_0\|_{L_p(\mathcal{R}_\theta^2)}+\|x_1\|_{L_p(\mathcal{R}_\theta^2)}.
\end{equation}
By the triangle inequality, it is reduced to proving that
\begin{equation}\label{Tmies}
  \|T_{m_{i,t}^\nu}(x_i)\|_{L_p(L_\infty([1,2])\overline{\otimes}\mathcal{R}_\theta^2))}\lesssim\|x_i\|_{L_p(\mathcal{R}_\theta^2)},\qquad i=0,1.
\end{equation}
Combining Lemma \ref{lem1} with the intertwining identity \eqref{ii}, we have
\begin{align*}
  \|T_{m_{i,t}^\nu}(x_i)\|_{L_p(L_\infty([1,2])\overline{\otimes}\mathcal{R}_\theta^2))}^p &=\int_1^2\|T_{m_{i,t}^\nu}(x_i)\|_{L_p(\mathcal{R}_\theta^2))}^p dt \\
 &=\int_1^2\big\|h_\varepsilon^{1/p}\big(\sigma_\theta(T_{m_{i,t}^\nu}(x_i))\big)\big\|_{L_p(\mathcal{N}_\theta)}^pdt\\
 &=\int_1^2\big\|h_\varepsilon^{1/p}(T_{m_{i,t}^\nu}\otimes id_{\mathcal{R}_\theta^2})(\sigma_\theta(x_i))\big\|_{L_p(\mathcal{N}_\theta)}^pdt\\&\leq I_\varepsilon+II_\varepsilon.
\end{align*}
Here
\begin{equation*}
  I_\varepsilon=\int_1^2\big\|(T_{m_{i,t}^\nu}\otimes id_{\mathcal{R}_\theta^2})(h_\varepsilon^{1/p}\sigma_\theta (x_i))\big\|_{L_p(\mathcal{N}_\theta)}^pdt
\end{equation*}
and
\begin{equation*}
  II_\varepsilon=\int_1^2\big\|h_\varepsilon^{1/p}(T_{m_{i,t}^\nu}\otimes id_{\mathcal{R}_\theta^2})(\sigma_\theta(x_i))-(T_{m_{i,t}^\nu}\otimes id_{\mathcal{R}_\theta^2})(h_\varepsilon^{1/p}\sigma_\theta (x_i))\big\|_{L_p(\mathcal{N}_\theta)}^pdt.
\end{equation*}
To estimate $I_\varepsilon$, we utilize Theorem \ref{thmloc}, specifically the inequalities \eqref{eq126} and \eqref{eq127}, along with Lemma \ref{lem1}. This leads to the following inequality:
\begin{align}\label{Ies} I_\varepsilon\lesssim \|h_\varepsilon^{1/p}(\sigma_\theta (x_i))\big\|_{L_p(\mathcal{N})}^p=\|x_i\|_{L_p(\mathcal{R}_\theta^2)}^p.
 \end{align}
On the other hand, we claim that
\begin{align}\label{IIes}
\lim_{\varepsilon\rightarrow0}II_\varepsilon=0.
 \end{align}
Based on this claim, the inequalities \eqref{Ies} and \eqref{IIes} imply the inequality \eqref{Tmies}, which in turn implies Theorem \ref{qeloc}.

 At the end of this section, we complete the proof of the claim \eqref{IIes}. Note that
  $T_{m_{i,t}^\nu}(h_\varepsilon^{1/p}\exp_\xi)=T_{m_{i,t}^\nu(\cdot+\xi)}(h_\varepsilon^{1/p})\exp_\xi$, we have
  \begin{align*}
    &h_\varepsilon^{1/p}(T_{m_{i,t}^\nu}\otimes id_{\mathcal{R}_\theta^2})(\sigma_\theta(x_i))-(T_{m_{i,t}^\nu}\otimes id_{\mathcal{R}_\theta^2})(h_\varepsilon^{1/p}\sigma_\theta (x_i))\\
    &=\int_{\mathbb{R}^2}\left(h_\varepsilon^{1/p} T_{m_{i,t}^\nu}(\exp_\xi)-T_{m_{i,t}^\nu}(h_\varepsilon^{1/p}\exp_\xi)\right)\otimes f_i(\xi)U(\xi)d\xi\\
    &=\int_{\mathbb{R}^2}\left(h_\varepsilon^{1/p} m_{i,t}^\nu(\xi)-T_{m_{i,t}^\nu(\cdot+\xi)}(h_\varepsilon^{1/p})\right)\exp_\xi \otimes f_i(\xi)U(\xi)d\xi.\\
\end{align*}
Then, by Lemma \ref{HY} and the Minkowski inequality, we have
\begin{align*}
   & \Big\|(h_\varepsilon^{1/p}(T_{m_{i,t}^\nu}\otimes id_{\mathcal{R}_\theta^2})(\sigma_\theta(x_i))-(T_{m_{i,t}^\nu}\otimes id_{\mathcal{R}_\theta^2})(h_\varepsilon^{1/p}\sigma_\theta (x_i))\Big\|_{L_p(\mathcal{N}_\theta)}^p\\
   &=\int_{\mathbb{R}^2}\left\|\int_{\mathbb{R}^2}\left(h_\varepsilon^{1/p} m_{i,t}^\nu(\xi)-T_{m_{i,t}^\nu(\cdot+\xi)}(h_\varepsilon^{1/p})\right)\exp_\xi f_i(\xi)U(\xi)d\xi\right\|_{L_p(\mathcal{R}_\theta^2)}^pd\eta\\
   &\leq\int_{\mathbb{R}^2}\Big(\int_{\mathbb{R}^2}\left|h_\varepsilon^{1/p} m_{i,t}^\nu(\xi)-T_{m_{i,t}^\nu(\cdot+\xi)}(h_\varepsilon^{1/p})\right|^{p^\prime}(\eta)|f_i|^{p^\prime}(\xi) d\xi\Big)^{p/{p^\prime}} d\eta\\
   &\leq\Big(\int_{\mathbb{R}^2}\Big(\int_{\mathbb{R}^2}\left|h_\varepsilon^{1/p} m_{i,t}^\nu(\xi)-T_{m_{i,t}^\nu(\cdot+\xi)}(h_\varepsilon^{1/p})\right|^{p}(\eta) d\eta\Big)^{{p^\prime}/p}|f_i|^{p^\prime}(\xi) d\xi\Big)^{p/{p^\prime}}\\
  &=\Big(\int_{\mathbb{R}^2}\Big(\int_{\mathbb{R}^2}\left|h_1^{1/p} m_{i,t}^\nu(\xi)-T_{m_{i,t}^\nu(\sqrt{\varepsilon}\cdot+\xi)}(h_1^{1/p})\right|^{p}(\eta) d\eta\Big)^{{p^\prime}/p}|f_i|^{p^\prime}(\xi) d\xi\Big)^{p/{p^\prime}} \\
  &\leq \Big(\int_{\mathbb{R}^2}\Big(\int_{\mathbb{R}^2}\left|\big(h_1^{1/p}\big)^{\wedge} (\zeta) \big(m_{i,t}^\nu(\xi)-m_{i,t}^\nu(\sqrt{\varepsilon}\zeta+\xi)\big)\right|^{p^\prime} d\zeta\Big)|f_i|^{p^\prime}(\xi) d\xi\Big)^{p/{p^\prime}},
     \end{align*}
 where the last inequality follows from the classical Hausdorff-Young inequality. Note that $(h_1^{1/p}\big)^{\wedge}$ and $f_i$ are Schwartz functions, by the dominated convergence theorem, we get
  \begin{equation*}
    \lim_{\varepsilon\rightarrow0} \int_1^2\Big(\int_{\mathbb{R}^2}\Big(\int_{\mathbb{R}^2}\left|\big(h_1^{1/p}\big)^{\wedge} (\zeta) \big(m_{i,t}^\nu(\xi)-m_{i,t}^\nu(\sqrt{\varepsilon}\zeta+\xi)\big)\right|^{p^\prime} d\zeta\Big)|f_i|^{p^\prime}(\xi) d\xi\Big)^{p/{p^\prime}}dt=0,
  \end{equation*}
  which implies that $\lim_{\varepsilon\rightarrow0}II_\varepsilon=0.$
\begin{remark}\label{rem:trans}
{\rm (i) The aforementioned transference technique remains valid for higher-dimensional spaces $\mathcal{R}_\theta^d$ $d\geq3$.

(ii) In the above approximation arguments, the application of the Hausdorff-Young inequalities necessitates $p\geq2$; the approximation arguments and thus the transference techniques become quite tricky for $1\leq p<2$, and it is still open although a weak version is already available in  \cite{Hong+}.
}
\end{remark}

\appendix
\section{Sharp endpoint $L_p$ estimates}\label{AP1}
 In this appendix, we delve into the Cauchy problem associated with the wave equation in the context of quantum Euclidean spaces. Throughout this appendix, we will use the following convention: The tensor von Neumann algebra $L_\infty(\mathbb{R}^d)\overline{\otimes}\mathcal{M}$ is denoted by $\mathcal{N}_d$, and for a reasonable symbel $m$, the associated operator $T_m$ on $\mathbb{R}^d$ (resp. $\mathcal{N}_d$)  is defined by
 \begin{equation*}
   T_m(f)(\eta):=(m(\cdot)\hat{f}(\cdot))^{\vee}(\eta),\quad f\in \mathcal{S}(\mathbb{R}^d)\ (\text{resp.}\ f\in \mathcal{S}(\mathbb{R}^d){\otimes}\mathcal{S}(\mathcal{M})).
 \end{equation*}
 And for $\xi\in\mathbb{R}^d$, we set $$m_{0,t}^\nu(\xi)=\frac{\cos(|t\xi|)}{(1+|\xi|^2)^{\frac{\nu}{2}}}, \quad m_{1,t}^\nu(\xi)=\frac{\sin(|t\xi|)}{|\xi|(1+|\xi|^2)^{\frac{\nu-1}{2}}}.$$

Considering the wave equation with the initial values $x_0,x_1\in\mathcal{S}(\mathcal{R}_\theta^d)$ on $\mathcal{R}_\theta^d,$
\begin{eqnarray}\label{weqq1}
\left\{
\begin{array}{ll}
   (\partial_{tt}-\Delta_\theta)u=0, \  t\in \mathbb{R}_+,  \\[5pt]
 u(0)=x_0,\\[5pt]\partial_tu(0)=x_1.
\end{array}
\right.
\end{eqnarray}
 Here, $\partial_tu$  is defined in the sense of:
\begin{equation*}
  \lim_{h\rightarrow0}\frac{u(s+h)-u(s)}{h}=\partial_tu (s)\quad \text{in}\ \mathcal{S}^\prime(\mathcal{R}_\theta^d).
\end{equation*}
The solution $u(t)$ to this equation can be represented as a sum of Fourier multiplier operators:
\begin{equation*}
  u(t)=c_0T_{m_{0,t}^0}(x_0)+c_1T_{m_{1,t}^1}(x_1),
\end{equation*}
where $c_0,c_1$ are two constants.
We present the following noncommutative analogue of the fixed-time $L_p$  estimate on $\mathcal{R}_\theta^d$:
\begin{theorem}\label{fix time}
  Let $1<p<\infty$, and let $u(t)$ be the solution to the Cauchy problem of the wave equation. Then, for any fixed time $0<t<\infty$, the following estimate
  \begin{equation*}
    \|u(t)\|_{L_p(\mathcal{R}_\theta^d)}\leq C_{p,t}\Big(\|x_0\|_{L_{p,\nu}(\mathcal{R}_\theta^d)}+ \|x_1\|_{L_{p,\nu-1}(\mathcal{R}_\theta^d)}\Big)
  \end{equation*}
  holds for all $\nu\geq s_p=(d-1)\big|\frac{1}{2}-\frac{1}{p}\big|$.
\end{theorem}
By leveraging the duality argument, it suffices to demonstrate Theorem \ref{fix time} for the range $2\leq p<\infty$. Subsequently, by applying the transference technique outlined in Section \ref{S6}, Theorem \ref{fix time} can be reduced to the following operator-valued fixed-time  $L_p$ estimate:
\begin{theorem}\label{fix time1}
  Let $1<p<\infty$ and  $i\in\{0,1\}$. For a fixed time $t$, the following estimate
  \begin{equation}\label{fix es1}
    \|T_{m_{i,t}^\nu}(f)\|_{L_p(\mathcal{N}_d)}\leq C_{p,t}\|f\|_{L_p(\mathcal{N}_d)}
  \end{equation}
  holds for all $\nu\geq s_p=(d-1)\big|\frac{1}{2}-\frac{1}{p}\big|$ and $f\in {L_p(\mathcal{N}_d)}$. Here $C_{p,t}$ is locally bounded with respect to $t$.
\end{theorem}
To prove Theorem \ref{fix time1}, we rely on the theory of operator-valued Hardy spaces introduced in \cite{Mei09}, but we omit the details here for simplicity.
\begin{lemma}[\cite{Mei09}, Theorem 6.4]\label{H1mul}
  Let $T_m$ be a Fourier multiplier bounded on the Hardy space $H_1(\mathbb{R}^d)$. Then $T_m$ automatically extends to a bounded Fourier multiplier on the operator-valued Hardy spaces $H^{cr}_1(\mathcal{N}_d)$. Moreover, one has
  \begin{equation*}
    \|T_m\|_{H^{cr}_1(\mathcal{N}_d)\rightarrow H^{cr}_1(\mathcal{N}_d)}\leq c_d\|T_m\|_{H_1(\mathbb{R}^d)\rightarrow H_1(\mathbb{R}^d)}.
  \end{equation*}
  \end{lemma}

  \begin{lemma}[\cite{Miyachi80}, Corollary 1]\label{opH1}
  Fix $t\in\mathbb{R}_+$. If $\nu\geq s_1$, one has
  \begin{equation*}
    \max\{\|T_{m_{0,t}^\nu}\|_{H_1(\mathbb{R}^d)\rightarrow H_1(\mathbb{R}^d)},\|T_{m_{1,t}^\nu}\|_{H_1(\mathbb{R}^d)\rightarrow H_1(\mathbb{R}^d)}\}\leq C_{t},
  \end{equation*}
  where $C_{t}$ is locally bounded with respect to $t$.

  \end{lemma}
  Now, let us proceed with the proof of Theorem \ref{fix time1}.
  \begin{proof}[\textbf{Proof of Theorem \ref{fix time1}}]
    Let us denote the operator-valued Sobolev space as  $L_{p,s}(\mathcal{N}_d)$, To establish Theorem \ref{fix time1}, it suffices to demonstrate the following inequality for $1<p<\infty$:
    \begin{equation}\label{fix es2}
      \|T_{m_{i,t}^{s_p}}(f)\|_{L_p(\mathcal{N}_d)}=\|T_{m_{i,t}^i}f\|_{L_{p,-s_p+i}(\mathcal{N}_d)}\leq C_{p,t}\|f\|_{L_p(\mathcal{N}_d)}, \quad i=0,1.
    \end{equation}
     We can prove the inequality \eqref{fix es2} using interpolation and duality: Combining Lemma \ref{opH1}, Lemma \ref{H1mul}, and the fact that $H^{cr}_1(\mathcal{N}_d)\subset L_1(\mathcal{N}_d)$, we obtain
    \begin{equation*}
      \big\|T_{m_{i,t}^i}f\big\|_{L_{1,-s_1+i}(\mathcal{N}_d)}=\|T_{m_{i,t}^{s_1}}f\big\|_{L_1(\mathcal{N}_d)}\leq \big\|T_{m_{i,t}^{s_1}}f\big\|_{H^{cr}_1(\mathcal{N}_d)}\leq C_t\|f\|_{H^{cr}_1(\mathcal{N}_d)}.
    \end{equation*}
    Applying the Plancherel theorem, we get
    \begin{equation*}
      \big\|T_{m_{i,t}^i}f\big\|_{L_{2,-s_2+i}(\mathcal{N}_d)}=\big\|T_{m_{i,t}^{s_2}}f\big\|_{L_{2}(\mathcal{N}_d)}\leq C_t\|f\|_{L_2(\mathcal{N}_d)}.
    \end{equation*}
    Then, by the complex interpolation, the estimate
     \begin{equation*}
       \|T_{m_{i,t}^i}f\|_{L_{p,-s_p+i}(\mathcal{N}_d)}\leq C_{p,t}\|f\|_{L_p(\mathcal{N}_d)}
     \end{equation*}
     holds for $1<p\leq2$. Using duality, we immediately obtain the inequality \eqref{fix es2} for $p\in(2,\infty)$.
     Thus, we have shown that Theorem \ref{fix time1} holds for all $1<p<\infty$, completing the proof.
  \end{proof}
\noindent {\bf Acknowledgements} \
This work is partially supported by the National Natural Science Foundation of China (No. 12071355, No. 12325105, No. 12031004, and No. W2441002).
X. Lai is supported by National Natural Science Foundation of China (No. 12322107, No. 12271124 and No. W2441002) and Heilongjiang Provincial Natural Science Foundation of China (YQ2022A005).

\end{document}